\documentclass[12pt,letterpaper]{amsart}
\usepackage{amsmath,txfonts}
\usepackage{amssymb}
\usepackage{amsxtra}
\usepackage{amsthm, color}
\usepackage{txfonts}
\usepackage{graphicx}
\usepackage{times}
\usepackage{citeref}
\usepackage{tikz}
\usepackage{hyperref}
\usepackage{stmaryrd}
\usepackage[T3,T1]{fontenc}
\usepackage{pgfplots}

\usetikzlibrary{calc}

\usepackage{mathrsfs}
\usepackage{amsfonts}
\usepackage{amssymb}
\usepackage{ifthen}
\usepackage{graphicx}
\nonstopmode \numberwithin{equation}{section}

\newtheorem{thm}{Theorem}[section]
\newtheorem{lem}{Lemma}[section]
\newtheorem{cor}[thm]{Corollary}
\newtheorem{prop}[thm]{Proposition}

\newtheorem{step}{Step}[section]

\theoremstyle{definition}
\newtheorem{mlem}{Main lemma}[section]
\newtheorem{assertion}{Assertion}[section]
\newtheorem{cl}{Claim}[section]
\newtheorem{ca}{Case}[section]
\newtheorem{sca}{Subcase}[section]
\newtheorem{scl}{Subclaim}[section]
\newtheorem{conj}[thm]{Conjecture}
\newtheorem{fact}{Fact}[section]
\newtheorem{defn}[thm]{Definition}
\newtheorem{op}[thm]{Open Problem}

\newtheorem{ques}{Question}[section]
\newtheorem{rem}[thm]{Remark}
\newtheorem{exam}[thm]{Example}

\numberwithin{equation}{section}

\newcounter {own}
\def\theown {\thesection       .\arabic{own}}

\newenvironment{pf}[1][]{%
 \vskip 3mm
 \noindent
 \ifthenelse{\equal{#1}{}}%
  {{\slshape Proof. }}%
  {{\slshape #1.} }%
 }%
{\qed\bigskip}

\newcounter{alphabet}
\renewcommand{\thealphabet}{\Alph{alphabet}}

\newenvironment{Thm}[1][]{\refstepcounter{alphabet}%
	\bigskip
	\noindent
	{\bf Theorem \thealphabet}%
	\ifthenelse{\equal{#1}{}}{}{ (#1)}%
	{\bf .} \itshape
}{\vskip 8pt}

\newenvironment{Lem}[1][]{\refstepcounter{alphabet}%
	\bigskip
	\noindent
	{\bf Lemma \thealphabet}%
	{\bf .} \itshape
}{\vskip 8pt}

\newcommand{\Aut}{{\operatorname{Aut}}}

\def\be{\begin{equation}}
\def\ee{\end{equation}}

\newcommand{\ben}{\begin{enumerate}}
\newcommand{\een}{\end{enumerate}}

\newcommand{\blem}{\begin{lem}}
\newcommand{\elem}{\end{lem}}
\newcommand{\bthm}{\begin{thm}}
\newcommand{\ethm}{\end{thm}}
\newcommand{\bcor}{\begin{cor}}
\newcommand{\ecor}{\end{cor}}
\newcommand{\beg}{\begin{exam}}
\newcommand{\eeg}{\end{exam}}
\newcommand{\begs}{\begin{examples}}
\newcommand{\eegs}{\end{examples}}
\newcommand{\bdefe}{\begin{defn}}
\newcommand{\edefe}{\end{defn}}
\newcommand{\bques}{\begin{ques}}
\newcommand{\eques}{\end{ques}}
\newcommand{\bei}{\begin{itemize}}
\newcommand{\eei}{\end{itemize}}
\newcommand{\bcon}{\begin{conj}}
\newcommand{\econ}{\end{conj}}
\newcommand{\bop}{\begin{op}}
\newcommand{\eop}{\end{op}}

\newcommand{\bas}{\begin{assertion}}
\newcommand{\eas}{\end{assertion}}

\newcommand{\bfa}{\begin{fact}}
\newcommand{\efa}{\end{fact}}

\newcommand{\bca}{\begin{ca}}
\newcommand{\eca}{\end{ca}}

\newcommand{\bst}{\begin{step}}
\newcommand{\est}{\end{step}}

\newcommand{\bsca}{\begin{sca}}
\newcommand{\esca}{\end{sca}}

\newcommand{\bcl}{\begin{cl}}
\newcommand{\ecl}{\end{cl}}

\newcommand{\bmlem}{\begin{mlem}}
\newcommand{\emlem}{\end{mlem}}

\newcommand{\bscl}{\begin{scl}}
\newcommand{\escl}{\end{scl}}

\newcommand{\bcons}{\begin{conjs}}
\newcommand{\econs}{\end{conjs}}

\newcommand{\bprop}{\begin{prop}}
\newcommand{\eprop}{\end{prop}}

\newcommand{\br}{\begin{rem}}
\newcommand{\er}{\end{rem}}
\newcommand{\brs}{\begin{rems}}
\newcommand{\ers}{\end{rems}}
\newcommand{\bo}{\begin{obser}}
\newcommand{\eo}{\end{obser}}
\newcommand{\bos}{\begin{obsers}}
\newcommand{\eos}{\end{obsers}}
\newcommand{\bpf}{\begin{pf}}
\newcommand{\epf}{\end{pf}}
\newcommand{\ba}{\begin{array}}
\newcommand{\ea}{\end{array}}
\newcommand{\beq}{\begin{eqnarray}}
\newcommand{\beqq}{\begin{eqnarray*}}
\newcommand{\eeq}{\end{eqnarray}}
\newcommand{\eeqq}{\end{eqnarray*}}

\newcounter{minutes}
\divide\time by 60
\newcounter{hours}
\multiply\time by 60 \addtocounter{minutes}{-\time}

\begin{document}
	\title{Hardy-Littlewood type phenomena and the Girela-Pel\'aez conjecture for the M\"obius invariant Laplacian operator}

	\author{Jiaolong Chen}
	\address{J. L. Chen, Key Laboratory of Computing and Stochastic Mathematics (Ministry of Education),
 School of Mathematics and Statistics, Hunan Normal University, Changsha, Hunan 410081, P. R. China}
	\email{jiaolongchen@sina.com}

\author{Shaolin Chen${}^{~\mathbf{*}}$}
\address{S. L. Chen,    Center for Applied Mathematics of Guangxi, Guangxi Normal University,
Guilin, Guangxi 541004, People's Republic of China} \email{mathechen@126.com}

\author{Hidetaka Hamada}
\address{H. Hamada, Industry-Academia Co-innovation and Research Promotion Headquarters,
Kyushu Sangyo University,
3-1 Matsukadai 2-Chome, Higashi-ku, Fukuoka 813-8503, Japan.}
\email{hi.hamada01@gmail.com}

	\author{Qianyun Li}
	\address{Q. Y. Li, Key Laboratory of Computing and Stochastic Mathematics (Ministry of Education),
School of Mathematics and Statistics, Hunan Normal University, Changsha, Hunan 410081, P. R. China}
	\email{liqianyun@hunnu.edu.cn}
	
	\keywords{Hardy-Littlewood type theorems, modulus of continuity, M\"obius invariant Laplacian.\\
		$^{\mathbf{*}}$Corresponding author}
	

	\maketitle
	
	\makeatletter\def\thefootnote{\@arabic\c@footnote}\makeatother

\begin{abstract}
The purpose of this paper is twofold. First, we investigate  the Hardy-Littlewood type phenomena for Dirichlet solutions to the M\"obius invariant Laplace equation
on the unit ball in $\mathbb{R}^n$.
Our work extends and improves several key results due to  Pavlov\'c [Rev. Mat. Iberoam. 23: 831-845, 2007] and Chen et al. [J. Geom. Anal. 34: 23 pp, 2024]. In particular, we give a complete answer to a question raised by Makoto Masumoto. Second, motivated by Aikawa's work, we study the boundedness of the operator norm of $P_{\alpha}$,
where $P_{\alpha}[\varphi]$ is the Dirichlet solution of such equation for the boundary data $\varphi$.
By using alternative proof techniques, we obtain an equivalent characterization of the boundedness of the operator norm of $P_{\alpha}$.
Finally, we show that the Girela-Pel\'aez conjecture holds positively for more general classes of functions induced by the M\"obius invariant Laplacian operator.
\end{abstract}
	
\maketitle \pagestyle{myheadings}
\markboth{J. L. Chen, S. L. Chen, H. Hamada and Q. Y. Li}{Hardy-Littlewood type phenomena and the Girela-Pel\'aez conjecture for the M\"obius invariant Laplacian operator}

	\section{Preliminaries and the statement of the main results}\label{csw-sec1}
	
	Let $\mathbb{R}^{n}$ denote the standard real vector space of dimension $n$, where $n$ is a positive integer.
For
$$\left(x=(x_{1},\ldots,x_{n}), y=(y_{1},\ldots,y_{n})\right)\in\mathbb{R}^{n}\times \mathbb{R}^{n},
$$
the standard
 Hermitian scalar product of $x$ and $y$, and the
Euclidean norm of $x$ are given by $$\langle x,y\rangle :=
\sum_{j=1}^nx_jy_j\;\;\mbox{and}\;\; |x|:={\langle
x,x\rangle}^{1/2}, $$ respectively.
Sometimes it is convenient to
identify each point $x=(x_{1},\ldots,x_{n})\in\mathbb{R}^{n}$ with an $n\times 1$ column matrix
so that
$$x=\left(\begin{array}{cccc}
x_{1}   \\
\vdots \\
 x_{n}
\end{array}\right).
$$
We denote the unit ball and the unit sphere in $\mathbb{R}^n$ by $\mathbb{B}^n$ and $\mathbb{S}^{n-1}$, respectively.
In particular, let $\mathbb{R}:=\mathbb{R}^{1}$, $\mathbb{D}:=\mathbb{B}^2$ and $\mathbb{T}:=\mathbb{S}^1$.
For a matrix $A = (a_{ij})_{n \times n} \in \mathbb{R}^{n \times n}$, the matrix norm is defined as
$$\|A\|=\sup\left\{|A\xi|:\; \xi\in \mathbb{S}^{n-1}\right\}.$$


For $x\in\mathbb{B}^{n}$, let $$\delta_{\alpha}(x)=(1-|x|^{2})^{\alpha}$$
be the standard weight on $\mathbb{B}^{n}$, where $\alpha\in\mathbb{R}$. We recall the
differential operator $\Delta_{\delta_{\alpha}}$ associated with the  weight $\delta_{\alpha}$,
defined by
\be\label{C-y}\Delta_{\delta_{\alpha}}:={\rm div}(\delta_{\alpha}^{-1}\nabla)+\alpha(n-2-\alpha)\delta_{\alpha+1}^{-1},\ee
where  $\nabla$ and ${\rm div}$ denote {\it the gradient} and {\it divergence operators}, respectively. For $n=2$,
the expression (\ref{C-y}) reveals a connection to the  conductivity equations studied by
Astala and P\"aiv\"arinta \cite{A-P}. 

A mapping $u \in C^{2}\left(\mathbb{B}^{n}\right)(n \geq 2)$, i.e., the set of all twice continuously differentiable functions of
$\mathbb{B}^{n}$, is said to be {\it $\alpha$-harmonic} if it satisfies {\it the  M\"{o}bius invariant  Laplacian equation}
$$\Delta_{\alpha}u=0,$$
where
\begin{equation*}
\Delta_{\alpha}:=(1-|x|^{2})^{\alpha+2}\Delta_{\delta_{\alpha}}=(1-|x|^{2})\left[(1-|x|^{2})\Delta+2\alpha\langle\nabla,x\rangle+\alpha(n-2-\alpha)\right],
\end{equation*}
$x=(x_{1},\ldots,x_{n})\in\mathbb{B}^{n}$ and
$$\Delta:=\sum_{j=1}^{n}\frac{\partial^{2}}{\partial\,x^{2}_{j}}$$ is the classical {\it Laplace operator} (see \cite{ABR-2001}).
Here, we refer to $\Delta_{\alpha}$ as {\it the M\"obius invariant Laplacian} because it satisfies the following invariance property (see \cite[Proposition 3.2]{Liu09}):
\begin{equation*}
		\Delta_{\alpha}\left\{\left(\operatorname{det} D\psi (x)\right)^{\frac{n-2- \alpha}{2 n}} u(\psi(x))\right\}=\left(\operatorname{det} D\psi (x)\right)^{\frac{n-2-\alpha}{2 n}}\left(\Delta_{\alpha} u\right)(\psi(x))
	\end{equation*}
for every $u \in C^{2}(\mathbb{B}^n)$ and every $\psi \in {\rm Aut}(\mathbb{B}^n)$,
where ${\rm Aut}(\mathbb{B}^n)$ denotes all M\"obius transformations  of $\mathbb{B}^{n}$ onto itself.
Obviously, $\Delta_{0}=\Delta$.
	


In particular, when $\alpha=0$ or $\alpha=n-2$, the mapping $u$ is referred to as {\it harmonic} or {\it hyperbolic harmonic}, respectively.
Let \be\label{eq-mo}\Delta_{h}:=\Delta_{n-2}.\ee
See, for example, \cite{ABR-2001,chen2021,chen2018,Sto-2016} for the properties of these classes of mappings.

The operator $\Delta_{\alpha}$ is closely connected with polyharmonic mappings (see \cite{AH2014, Liu21}) and with solutions to the Weinstein equation (see \cite{leu}). For a detailed treatment of the fundamental properties of $\Delta_{\alpha}$ in the unit disk $\mathbb{D}$, we refer to \cite{AH2014, chen23, chen15, KMM2021, Ol14, Ol20}. For extensions to higher dimensions, see \cite{chen24, li24, Liu04, Liu09, Liu21, Liu24, ZHD24, zhou22} and the references therein.


	In \cite{Liu04}, Liu and Peng  investigated  the solvability of the
	Dirichlet problem associated with the (M\"{o}bius) invariant Laplacian as  follows:
	\begin{equation}\label{eq-1.1}
		\left\{\begin{array}{ll}
			\Delta_{\alpha} u(x)=0, & x\in\mathbb{B}^{n}, \\
			u(\zeta)=\phi(\zeta),& \zeta\in\mathbb{S}^{n-1}.
		\end{array}\right.
	\end{equation}
	They showed that if   $\phi\in C ( \mathbb{S}^{n-1})$, i.e., the set of all  continuous functions of
$\mathbb{S}^{n-1}$, then the Dirichlet problem \eqref{eq-1.1}
	has a solution  if and only if $\alpha>-1$ (see \cite[Theorem 2.4]{Liu04}).
	In this case the solution is unique and   can be expressed by the following Poisson type integral:
	\begin{align*}
		u(x)&=\int_{\mathbb{S}^{n-1}} P_{\alpha}(x, \zeta) \phi(\zeta) d \sigma(\zeta)={P}_{\alpha}[\phi](x),	
	\end{align*}
	where $d \sigma$  denotes the normalized surface measure on  $\mathbb{S}^{n-1}$   so that $\sigma(\mathbb{S}^{n-1})=1$,
	and
	\begin{align*}
		P_{\alpha}(x, \zeta)&=C_{n, \alpha} \frac{\left(1-|x|^{2}\right)^{1+ \alpha}}{|x-\zeta|^{n+\alpha}}
		\;\;\text{with}\;\;C_{n, \alpha}=\frac{\Gamma\left(\frac{n+\alpha}{2}\right) \Gamma(1+\frac{\alpha}{2})}{\Gamma\left(\frac{n}{2}\right) \Gamma(1+\alpha)}.
	\end{align*}
	Note that $P[\phi]:=P_0[\phi]$ is the usual {\it Poisson integral} of $\phi$
	and is harmonic in $\mathbb{B}^n$.

For convenience, we make a notational convention.
Throughout this paper,
 we use the symbol $C$ to denote various positive
constants, whose values may change from one occurrence to another.
Also we denote by $C=C_{a_{1},a_{2},\ldots}$ a constant that depends only on the given parameters $a_{1}$, $a_{2}$, $\ldots$
and whose value may vary from one occurrence to another.

	\subsection*{The invariant measure of $SO(n)$}
 \begin{defn}
 The set $SO(n) \subseteq O(n)$ of special orthogonal matrices is defined as
$$
		SO(n) := \{ U \in O(n) : \det(U) = 1 \},
$$
where $O(n)$ denotes the set of $n \times n$ real orthogonal matrices and each $U\in SO(n)$ is called a  rotation matrix.
	\end{defn}
	
	Let $D(n)$ be the  set of  $\{\theta_{i, j+1}:  1 \leq i \leq j \leq n-1\}$ with
$$
	0 \leq \theta_{1, j+1} < 2\pi \quad (1\leq j\leq n-1), \quad 0 \leq \theta_{i, j+1} \leq \pi \quad (2 \leq i \leq j \leq n-1).
$$
The above angles  are called Euler angles.
Under the Euler angle parameterization, each rotation matrix   corresponds to a set of angles $\{\theta_{j,k}\}$ (see \cite[p. 5]{DF16}).
	
In \cite{Hur1897},	Hurwitz introduced an invariant measure $d\mu$ for the orthogonal group $SO(n)$.
 Later, Diaconis and Porrester \cite{DF16} provides an explicit representation formula for it in terms of the Euler angles, i.e.,
	\begin{equation}\label{eq-1.1.3}
		d\mu  = 2^{n(n-1)/4} \prod_{1 \leq j < k \leq n}  ( \sin \theta_{j,k}  )^{j-1} d\theta_{j,k}.
	\end{equation}
	
	The volume of $SO(n)$, denoted by $\operatorname{vol}(SO(n))$, is defined as the integral of \eqref{eq-1.1.3} over the allowed range of Euler angles in $D(n)$. As noted by Hurwitz \cite{Hur1897}, performing this calculation yields
$$
	\operatorname{vol}(SO(n)) = \frac{1}{2} 2^{n(n+3)/4} \prod_{k=1}^{n} \frac{\pi^{k/2}}{\Gamma(k/2)}.
$$
	
	\subsection*{The $\Lambda_{\omega, p}(\mathbb{B}^{n})$ spaces} 

A continuous increasing function $\omega: [0, \infty) \to [0, \infty)$ with $\omega(0) = 0$ is called a {\it majorant} if $\omega(t)/t$ is non-increasing for $t>0$ (see \cite{DK1997, DK04}). According to \cite[p.147]{DK04}, a majorant $\omega$ is said to be {\it fast} if there exist some $\delta_0 > 0$ and a positive constant $C$ such that
$$\int_{0}^{\delta} \frac{\omega(t)}{t} d t \leqslant C \omega(\delta), \quad\mbox{for all}~ 0<\delta<\delta_{0},$$
and {\it slow} if there exist some $\delta_0 > 0$ and a positive constant $C$ such that
$$\delta \int_{\delta}^{\infty} \frac{\omega(t)}{t^{2}} d t \leqslant C \omega(\delta),
\quad \mbox{for all}~ 0<\delta<\delta_{0}.$$
A majorant $\omega$ is said to be {\it regular} if it is both fast and slow.

	Given a majorant  $\omega$  and a subset  $\Omega$  of  $\mathbb{R}^{n}$, a function  $f$  of  $\Omega$  into  $\mathbb{R}$   is said to belong to the {\it Lipschitz space}  $\Lambda_{\omega}(\Omega)$  if there is a positive constant  $C$  such that
	\begin{equation}\label{eq-CCHL-1}
	|f(x)-f(y)| \leq C \omega(|x-y|), \quad x, y \in \Omega .
	\end{equation}
	Furthermore, let
	$$
	\|f\|_{\Lambda_{\omega}(\Omega),s}:=\sup _{x, y \in \Omega, x \neq y} \frac{|f(x)-f(y)|}{\omega(|x-y|)}<\infty
	$$
	and
	\[
	\|f\|_{\Lambda_{\omega}(\Omega)}:=\sup_{x\in \Omega}|f(x)|+\|f\|_{\Lambda_{\omega}(\Omega),s}.
	\]
	Note that if  $\Omega$  is a proper subdomain of  $\mathbb{R}^{n}$  and  $f \in \Lambda_{\omega}(\Omega)$, then  $f$  is continuous on  $\overline{\Omega}$  and
 (\ref{eq-CCHL-1})  holds for $ x, y \in \overline{\Omega}$  (see \cite[p. 1]{DK04}).
	
	 Furthermore, we denote by  $\Lambda_{\omega, p}(\mathbb{B}^{n})$   the class of all Borel functions  $f: \mathbb{B}^{n} \to \mathbb{R}$ satisfying
$$\mathcal{L}_{p}[f]\left(x_{1}, x_{2}\right) \leqslant C \omega\left(\left|x_{1}-x_{2}\right|\right),~\mbox{for all}~x_{1}, x_{2} \in \mathbb{B}^{n},$$ where $C$ is a positive constant and $\mathcal{L}_{p}[f](x_1, x_2)$ is defined by
	$$
	\mathcal{L}_{p}[f]\left(x_{1}, x_{2}\right)=\left\{\begin{array}{ll}
		\left(\int_{SO(n)}\left|  f\left(Rx_{1}\right)-f\left(Rx_{2}\right)\right|^{p} d\mu(R)\right)^{\frac{1}{p}}, & \text { if } p \in(0, \infty), \\
		\left|f\left(x_{1}\right)-f\left(x_{2}\right)\right|, & \text { if } p=\infty.
	\end{array}\right.
	$$
	Here $R\in SO(n)$, and the integral over  $SO(n)$ is given explicitly by
 $$
 \int_{SO(n)}  f\left(Rx\right) d\mu(R)
 =\int_{K}
  f(R(\theta_{j,k})x)2^{n(n-1)/4} \prod_{1 \leq j < k \leq n} ( \sin \theta_{j,k}  )^{j-1}
  d\theta_{j,k},
 $$
where  $K$ is the domain of integration for the Euler angles parameterizing $SO(n)$.
	In particular, for $n=2$, the integral over $SO(2)$ evaluates as follows:
 $$\int_{SO(n)}d\mu(R)=\int_{0}^{2\pi}\sqrt{2}d\theta=2\sqrt{2}\pi.$$
	
	The {\it Lipschitz constant} of  $f \in \Lambda_{\omega, p}(\mathbb{B}^{n})$ is defined as follows:
	$$
	\|f\|_{\Lambda_{\omega, p}(\mathbb{B}^{n}),s}:=\sup _{x_{1}, x_{2} \in \mathbb{B}^{n}, x_{1} \neq x_{2}} \frac{\mathcal{L}_{p}[f]\left(x_{1}, x_{2}\right)}{\omega\left(\left|x_{1}-x_{2}\right|\right)}<\infty.
	$$
	
	Obviously,  $\|f\|_{\Lambda_{\omega, \infty}(\mathbb{B}^{n}), s}=\|f\|_{\Lambda_{\omega}(\mathbb{B}^{n}), s}$  and  $\Lambda_{\omega, \infty}(\mathbb{B}^{n})=\Lambda_{\omega}(\mathbb{B}^{n})$. Moreover, we define the space  $\Lambda_{\omega, p}(\mathbb{S}^{n-1})$, consisting of  $f \in L^{p}(\mathbb{S}^{n-1})$  for which
	$$
	\mathcal{L}_{p}[f]\left(\xi_{1}, \xi_{2}\right) \leqslant C \omega\left(\left|\xi_{1}-\xi_{2}\right|\right), \quad \xi_{1}, \xi_{2} \in \mathbb{S}^{n-1},
	$$
	where  $C>0$  is a constant and
	$$
	\mathcal{L}_{p}[f]\left(\xi_{1}, \xi_{2}\right)=\left\{\begin{array}{ll}
		\left(\int_{SO(n)}\left|f\left( R\xi_{1}\right)-f\left( R\xi_{2}\right)\right|^{p} d \mu(R)\right)^{\frac{1}{p}}, & \text { if } p \in(0, \infty), \\
		\left|f\left(\xi_{1}\right)-f\left(\xi_{2}\right)\right|, & \text { if } p=\infty .
	\end{array}\right.
	$$
	Furthermore, for $f \in \Lambda_{\omega, p}(\mathbb{S}^{n-1})$, let
	$$
	\|f\|_{\Lambda_{\omega, p}(\mathbb{S}^{n-1}),s}:=\sup _{x_{1}, x_{2} \in \mathbb{S}^{n-1}, x_{1} \neq x_{2}} \frac{\mathcal{L}_{p}[f]\left(x_{1}, x_{2}\right)}{\omega\left(\left|x_{1}-x_{2}\right|\right)}<\infty
	$$
	be the Lipschitz constant and
	\[
	\|f\|_{\Lambda_{\omega, p}(\mathbb{S}^{n-1})}:=\sup_{x\in \mathbb{S}^{n-1}}|f(x)|+\|f\|_{\Lambda_{\omega,p}(\mathbb{S}^{n-1}),s}.
	\]

	\subsection*{Hardy-Littlewood type phenomenon for the M\"obius invariant Laplacian operator}
Before presenting our main results, we first recall the classical Hardy-Littlewood theorem for complex-valued harmonic functions (see \cite{HL31,HL32}).

\begin{Thm}\label{Thm-A}
If $\varphi \in \Lambda_{\omega_{\beta}, \infty}(\mathbb{T})$, then $P[\varphi] \in \Lambda_{\omega_{\beta}, \infty}(\overline{\mathbb{D}})$, where $\beta \in (0,1)$ and $\omega_{\beta}(t) = t^{\beta}$ for $t \geqslant 0$.
\end{Thm}

Nolder and Oberlin later generalized Theorem \ref{Thm-A} by establishing a Hardy-Littlewood theorem for differentiable majorants (see \cite[Lemma 3.2]{NO88}). Subsequently, Dyakonov further extended Theorem \ref{Thm-A} to complex-valued harmonic functions on $\mathbb{D}$ as follows.

\begin{Thm}{\rm (\cite[Lemma 4]{DK1997})}\label{Thm-B}
Let $\omega$ be a regular majorant. If $\varphi \in \Lambda_{\omega, \infty}(\mathbb{T})$, then $P[\varphi] \in \Lambda_{\omega, \infty}(\overline{\mathbb{D}})$.
\end{Thm}

In a similar spirit to Theorem \ref{Thm-B}, Pavlovi\'c established the following Hardy-Littlewood theorem for harmonic functions on the unit ball $\mathbb{B}^{n}$.

\begin{Thm}{\rm (\cite[Theorem 1]{Pav})}\label{Thm-C}
Let $\omega$ be a slow majorant, and let $u$ be a harmonic function from ${\mathbb{B}^{n}}$ into $\mathbb{R}$
which is continuous on $\overline{\mathbb{B}^{n}}$. If $|u| \in \Lambda_{\omega, \infty}(\mathbb{S}^{n-1})$, then $u \in \Lambda_{\omega, \infty}(\mathbb{B}^{n})$.
\end{Thm}

In view of the proof of  \cite[Lemma 3.1]{DK06}, a function $u$ belongs to $\Lambda_{\omega, \infty}(\mathbb{S}^{n-1})$ if and only if $|u|$ also belongs to this space. Therefore, the condition ``$|u| \in \Lambda_{\omega, \infty}(\mathbb{S}^{n-1})$'' in Theorem C can be replaced by ``$u \in \Lambda_{\omega, \infty}(\mathbb{S}^{n-1})$'' 
as follows.

\begin{Thm}\label{Thm-C2}
Let $\omega$ be a slow majorant, and let $u$ be a harmonic function from ${\mathbb{B}^{n}}$ into $\mathbb{R}$
which is continuous on $\overline{\mathbb{B}^{n}}$. If $u\in \Lambda_{\omega, \infty}(\mathbb{S}^{n-1})$, then $u \in \Lambda_{\omega, \infty}(\mathbb{B}^{n})$.
\end{Thm}

Recently, Chen et al. generalized Theorem \ref{Thm-B} to $\alpha$-harmonic functions on $\mathbb{B}^n$ as follows.

\begin{Thm}{\rm (\cite[Theorems 2.1 and 2.3]{chen24})}\label{Thm-D}
Let $n \geq 2$, $\alpha \in (-1, \infty)$, and $\omega$ be a majorant. Then the following statements are equivalent:
\begin{enumerate}
\item[{\rm ($\mathcal{A}{1}$)}] If $\varphi \in \Lambda_{\omega, \infty}(\mathbb{S}^{n-1})$, then $P_{\alpha}[\varphi] \in \Lambda_{\omega, \infty}(\overline{\mathbb{B}^{n}})$.
\item[{\rm ($\mathcal{A}{2}$)}] There exists a positive constant $C$ such that for all $\delta \in (0, \pi]$,
\be\label{eq-CCHL-02}
\delta^{1+\alpha} \int_{\delta}^{\pi} \frac{\omega(t)\sin^{n-2}t}{t^{n+\alpha}} dt \leqslant C \omega(\delta).
\ee
\end{enumerate}
\end{Thm}

At the Workshop ``Prospects of Theory of Riemann Surfaces'' held in 2025 at Aichi Institute of Technology,  Makoto Masumoto posed the following question.

\begin{ques}\label{Qes-1}

Can the condition ($\mathcal{A}{2}$) in Theorem \ref{Thm-D} be replaced by a condition independent of the dimension $n$?
\end{ques}

A notable generalization of the Hardy-Littlewood theorem was later obtained by Chen and Hamada for complex-valued harmonic functions $f \in \Lambda_{\omega, p}(\mathbb{D})$, which they stated as follows.

\begin{Thm}{\rm (\cite[Theorem 2.11]{chen24B})}\label{Thm-E}
Let $\omega$ be a majorant and $p \in [1, \infty]$.
\begin{enumerate}
\item[{\rm ($\mathcal{B}{1}$)}] If $\varphi \in \Lambda_{\omega, p}(\mathbb{T})$ and $\varphi$ is continuous on $\mathbb{T}$, then $P[\varphi] \in \Lambda_{\omega, p}(\overline{\mathbb{D}})$.
\item[{\rm ($\mathcal{B}_{2}$)}] There exists a positive constant $C$ such that for all $\delta \in (0, \pi]$,

$$\delta \int_{\delta}^{\pi} \frac{\omega(t)}{t^{2}} d t \leqslant C \omega(\delta).$$
\end{enumerate}
Then, $(\mathcal{B}{2}) \Rightarrow (\mathcal{B}{1})$ for $p \in [1, \infty]$, and $(\mathcal{B}{1}) \Leftrightarrow (\mathcal{B}{2})$ for $p = \infty$.
\end{Thm}

Given Theorem \ref{Thm-E}, it is natural to ask the following question. However, we are currently unable to prove it here.

\begin{ques}(The Hardy–Littlewood type equivalent problem)\label{Qes-2}
Does $(\mathcal{B}_{1}) \Leftrightarrow (\mathcal{B}_{2})$ hold for all $p \in [1,\infty]$ in Theorem \ref{Thm-E}?
\end{ques}

The first aim of this paper is to present an improved and extended version of Theorems \ref{Thm-A}-\ref{Thm-E}
on ${\mathbb{B}^{n}}$. In particular, we show that condition \eqref{eq-CCHL-02} can be replaced by the inequality
\be\label{eq-CCHL-03}
\delta^{1+\alpha} \int_{\delta}^{\pi} \frac{\omega(t)}{t^{2+\alpha}} dt \leqslant C \omega(\delta),~\delta \in (0, \pi],
\ee
which is independent of $n$. This gives a positive answer to Question \ref{Qes-1}.

    \begin{thm}\label{thm-1.1}
    	Let  $n \geq 2$, $\alpha \in(-1, \infty)$,   $\omega$  be a majorant  and $p \in[1, \infty]$  be a constant.
    	\begin{enumerate}
    	\item[{\rm $\left(\mathcal{C}_{1}\right)$}]
    	  If  $\varphi \in \Lambda_{\omega, p}(\mathbb{S}^{n-1})$  and $\varphi$  is continuous on  $\mathbb{S}^{n-1}$, then  $P_{\alpha}[\varphi] \in \Lambda_{\omega, p}(\overline{\mathbb{B}^{n}})$.
    	\item[{\rm $\left(\mathcal{C}_{2}\right)$}]  There exists a positive constant $C$ such that inequality (\ref{eq-CCHL-03}) holds for all $\delta \in (0, \pi]$.
    	\end{enumerate}
    Then,  $\left(\mathcal{C}_{2}\right) \Rightarrow\left(\mathcal{C}_{1}\right)$  for  $p \in[1, \infty]$, and
    	$\left(\mathcal{C}_{2}\right) \Leftrightarrow\left(\mathcal{C}_{1}\right)$  for  $p=\infty$.
    \end{thm}

\begin{remark}
(1)
If $\alpha>0$, then every majorant satisfies $(\ref{eq-CCHL-03})$ (cf. \cite[Remark (2)]{chen24}).

(2)
If $\alpha=0$ and $\omega$ is slow, then $\omega$ satisfies $(\ref{eq-CCHL-03})$.

(3)
If $\alpha\in (-1,0)$,
then $\omega=\omega_{\beta}$ ($0<\beta\leq 1$) satisfies $(\ref{eq-CCHL-03})$
if and only if $\beta<1+\alpha$,
where $\omega_{\beta}(t) = t^{\beta}$ for $t \geqslant 0$.

\end{remark}

    	
    	
    Take majorants $\omega_1$ and $\omega_2$, and define the operator norm
    \[
    \|P_{\alpha}\|_{\omega_{1}\to\omega_{2}}=\sup_{\substack{f\in\Lambda_{\omega_{1}}(\mathbb{S}^{n-1})\\ \|f\|_{\Lambda_{\omega_{1}}(\mathbb{S}^{n-1})}\neq 0}}\frac{\|P_{\alpha}f\|_{\Lambda_{\omega_{2}}(\mathbb{B}^{n})}}{\|f\|_{\Lambda_{\omega_{1}}(\mathbb{S}^{n-1})}}.
    \]
    The finiteness of $\|P_{\alpha}\|_{\omega_1 \to \omega_2}$ is of particular interest.

For each $a \in \mathbb{S}^{n-1}$, we define a test function $\tau_{a, \omega}$ on $\mathbb{S}^{n-1}$ by
\begin{equation}\label{eq-1.9.1}
\tau_{a, \omega}(\xi) = \omega(|\xi - a|), \quad \text{for } \xi \in \mathbb{S}^{n-1}.
\end{equation}


      We shall see in Lemma \ref{lem-3.1} that  $\tau_{a, \omega} \in \Lambda_{\omega}(\mathbb{S}^{n-1})$.


Let $D$ be a bounded  domain in $\mathbb{R}^{n}$ with $n\geq2$.
By the classical Wiener criterion (see, e.g., \cite[Sec. 7.7]{AG}), if
$D$ is regular for the Dirichlet problem, then
$D$ has no trivial boundary points.
For the definition of a trivial boundary point, we refer the reader to \cite[Def. 1]{Ai2002}.
For a function $f$
on $\partial D$, we denote by $\mathcal{H}^{D}f$ the Dirichlet solution, for the Laplacian equation $\Delta_{0}\mathcal{H}^{D}f=0$ (or $\Delta\mathcal{H}^{D}f=0$), of $f$
over $D$, that is, $\mathcal{H}^{D}f$ is harmonic in $D$ and $\mathcal{H}^{D}f=f$
on $\partial D$. 
Recently, Aikawa \cite{Ai2002} studied the H\"older continuity of $\mathcal{H}^{D}f$. This work was later generalized in \cite{Ai2010} by
considering an arbitrary modulus of continuity $\mathcal{M}$, defined as the class of positive, nondecreasing, and concave functions $\phi$ on $(0, \infty)$ satisfying
$$\phi(0) = \lim_{t \to 0^{+}} \phi(t) = 0.$$ We now state the relevant theorem from \cite{Ai2010}.


      \begin{Thm}{\rm (\cite[Theorem 1.1]{Ai2010})}\label{Thm-F}
      	Let $D$ be a bounded regular domain in $\mathbb{R}^n$ and let $\psi \in \mathcal{M}$. Then the following statements are equivalent:
      	
      \begin{enumerate}
    	\item[{\rm $(i)$}]
      	  $\left\|\mathcal{H}^{D}\right\|_{\psi \rightarrow \psi}
     =\sup \left\{ \frac{\|\mathcal{H}^{D} f\|_{\psi,  D} }{\|f\|_{\psi, \partial D}}:f\in \Lambda_{\psi}(\partial D),\|f\|_{\psi,\partial D}\not=0\right\}<\infty$;
      	\item[{\rm $(ii)$}] there is a constant  $C \geqslant 1$  such that
      	$$
      	\mathcal{H}^{D} \tau_{a, \psi}(x) \leqslant C \psi(|x-a|) \quad \text { for } x \in D,
      	$$
      	whenever  $a \in \partial D$.
        \end{enumerate}
      \end{Thm}
\noindent
Here, we say $f\in \Lambda_{\psi}(E)$ for $\psi \in \mathcal{M}$ and for an arbitrary set $E\subset \mathbb{R}^{n}$ if $f$ is a bounded continuous function on $E$ with
$$
	\|f\|_{\psi,E}:=\sup_{x\in E}|f(x)|
+\sup_{\substack{x,y\in E\\ x\not=y}}
\frac{|f(x)-f(y)|}{\psi(|x-y|)}
<\infty
$$
and the function $\tau_{a, \psi} $ on $\partial D$ is defined by
$$
\tau_{a, \psi}(\xi)=\psi(|\xi-a|).
$$

      In \cite{It2012}, Itoh established related results for $p$-harmonic functions.
By analogy with Theorem \ref{Thm-F}, the second aim of this paper is to study the modulus of continuity of Dirichlet solutions to the M\"obius invariant Laplace equation on ${\mathbb{B}^{n}}$.
In the proofs of the main results in \cite{Ai2002} and \cite{Ai2010}, the authors used the local Poisson integral representation for harmonic functions.
 However, a local integral representation is not available for solutions $P_{\alpha}[\varphi]$ of the invariant Laplace equation when $\alpha \neq 0$, where $\varphi$ denotes the boundary data.
 Consequently, alternative methods of proof are necessary to investigate the boundedness of the operator norm of $P_{\alpha}$. Employing different proof techniques, we derive the following result.

     \begin{thm}\label{thm-1.2}
      Let  $\omega$  be a majorant and let $\alpha>-1$. Then the following statements are equivalent:

      \begin{enumerate}
    	\item[{\rm $(i)$}]  $\left\|P_{\alpha}\right\|_{\omega\to\omega}<\infty$.

      \item[{\rm $(ii)$}] there is a constant  C  such that
      $$
      P_{\alpha} [\tau_{a, \omega}](x) \leq C \omega(|x-a|) \quad \text { for } x \in \mathbb{B}^{n},
      $$
      whenever  $a \in \mathbb{S}^{n-1}$.
      \end{enumerate}
      \end{thm}

      \subsection*{The Girela-Pel\'aez conjecture for the M\"obius invariant Laplacian operator}
      For $n\geq2$ and $a\in\mathbb{B}^{n}$, the {\it M\"obius transformation} in
$\mathbb{B}^{n}$ is defined by
\be\label{eq-ex1}
\phi_{a}(x)=\frac{|x-a|^{2}a+(1-|a|^{2})(a-x)}{|x-a|^{2}+(1-|a|^{2})(1-|x|^{2})},~x\in\mathbb{B}^{n}.
\ee
We recall the following facts from
\cite{Bea}: For $a\in\mathbb{B}^{n}$ and
$\phi_{a}\in\Aut(\mathbb{B}^{n})$, we have $\phi_{a}(0)=a$,
$\phi_{a}(a)=0$, $\phi_{a}(\phi_{a}(x))=x \in\mathbb{B}^{n}$.
The {\it M\"obius invariant measure} $V_{h}$ on $\mathbb{B}^{n}$
is given by $$dV_{h}(x)=\frac{dV_{N}(x)}{(1-|x|^{2})^{n}},$$
where $dV_{N}$ denotes the  normalized Lebesgue volume measure on $\mathbb{B}^{n}$ such that $V_{N}(\mathbb{B}^{n})=1$.
%

Recall that a function is subharmonic with respect to the M\"obius invariant
Laplacian  as follows.
For $n\geq2$, an upper semicontinuous function
$f:~\mathbb{B}^{n}\mapsto[-\infty,\infty)$, with $f\not\equiv-\infty$, is {\it invariant subharmonic} in $\mathbb{B}^{n}$ if
\be\label{ho-1}f(a)\leq\int_{\mathbb{S}^{n-1}}f(\phi_{a}(r\zeta))d\sigma(\zeta)\ee
for all $a\in\mathbb{B}^{n}$ and all  sufficiently small $r\in(0,1)$,
where $\phi_{a}\in\Aut(\mathbb{B}^{n})$ is as in \eqref{eq-ex1}.
It is well known that $f\in C^{2}(\mathbb{B}^{n})$,  i.e., the set of all twice continuously differentiable functions of  $\mathbb{B}^{n}$, is  invariant subharmonic in $\mathbb{B}^{n}$
if and only if $\Delta_{h}f(x)\geq0$ for all $x\in\mathbb{B}^{n}$.
Moreover, $f\in C^{2}(\mathbb{B}^{n})$ is invariant harmonic in $\mathbb{B}^{n}$ if and only if equality holds in (\ref{ho-1}).
This is the so-called the well-known mean-value property with respect to the M\"obius invariant
Laplacian operator (\ref{eq-mo}) (see \cite{Kur,Sto-2016}).

Let $n\geq2$ and $\nu\geq1$ be integers.
For
$f:~\mathbb{B}^{n}\rightarrow\mathbb{R}^{\nu}$ such that $f$ is measurable,
let
$$
M_{p}(r,f)=\left(\int_{\mathbb{S}^{n-1}}|f(r\zeta)|^{p}\,d\sigma(\zeta)\right)^{\frac{1}{p}}~\mbox{and}~M_{\infty}(r,f)=\sup_{\zeta\in\mathbb{S}^{n-1}}|f(r\zeta)|.$$

      A classical result of Hardy and Littlewood asserts that if
$p\in(0,\infty]$, $\alpha\in(1,\infty)$ and $f$ is a holomorphic
function in $\mathbb{D}$, then (cf. \cite{CPR,CS-2015, GPP,GP,HL31,HL32,SS-05})
$$ M_{p}(r,f')=O \left(\left(\frac{1}{1-r}\right)^{\alpha} \right ) ~\mbox{ as $r\rightarrow1^{-}$}
$$
if and only if
$$M_{p}(r,f)=O \left (\left(\frac{1}{1-r}\right)^{\alpha-1}\right) ~\mbox{ as $r\rightarrow1^{-}$}.
$$
In \cite[Theorem 1(a)]{GP}, Girela and
Pel\'{a}ez refined the above result for the case $\alpha=1$ as follows.
If $p\in(2,\infty)$ and $f$ is a holomorphic function in
$\mathbb{D}$ such that
$$M_{p}(r,f')=O \left (\frac{1}{1-r}\right  )  ~\mbox{ as
$r\rightarrow1^{-}$},
$$
then for all $ \beta>\frac{1}{2}$,
      \be\label{eq1.1ga} M_{p}(r,f)=O \left
(\left(\log\frac{1}{1-r}\right)^{\beta} \right ) ~\mbox{ as
$r\rightarrow1^{-}$}.  \ee

In \cite[p.464, Equation (26)]{GP}, Girela and Pel\'aez conjectured that $\beta = 1/2$ in (\ref{eq1.1ga}).
Using the closed graph theorem, Girela et al. \cite{GPP} confirmed this conjecture by proving the following result.

     \begin{Thm} $($\cite[Theorem 1.1]{GPP}$)$\label{Thm-JA} Let $p\in(2,\infty)$. If  $f$ is
holomorphic in $\mathbb{D}$ such that
$$M_{p}(r,f')=O \left ( \frac{1}{1-r} \right  )  ~\mbox{ as $r\rightarrow1^{-}$},
$$
then
$$ M_{p}(r,f)=O \left (\left(\log\frac{1}{1-r}\right)^{\frac{1}{2}} \right ) ~\mbox{ as $r\rightarrow 1^{-}$}.
$$
\end{Thm}

By using weight,  Chen and Hamada \cite[Theorem 3]{CH2025MathZ} provided a close
relationship between the integral means of pluriharmonic (holomorphic) functions and
those of their derivatives in bounded symmetric domains $\Omega$ in $\mathbb{C}^n$.
In particular, a generalization of Theorem \ref{Thm-JA} to pluriharmonic functions on $\Omega$ was obtained.

     Let $u$ be a twice continuously differentiable functions of  $\mathbb{B}^{n}$ into $\mathbb{R}$. If there exist
real-valued nonnegative continuous functions $\chi_{1}$ in $\mathbb{B}^{n}$ and
nonnegative constants $\gamma_{j}~(j\in\{1,2,3\})$ such that
\be\label{Heinz-57}|u(x)|^{\gamma_{1}}|\Delta u(x)|^{\gamma_{2}}\leq\chi_{1}(x)|\nabla u(x)|^{\gamma_{3}},\ee
then we call $u$ satisfying {\it the Heinz type nonlinear
differential inequality}. By using (\ref{Heinz-57}), Heinz \cite{HZ} investigated the  mean curvature of surfaces,
the elliptic Monge-Amp${\rm\grave{e}}$re equations, the Poisson equations, the gradient equations, the Harnack inequality for a class of elliptic
differential equations and the existence of solutions to some class of elliptic
differential equations. This remarkable paper has attracted much attention of many mathematicians (see \cite{C2026,Mar,Ni,Qiu,Wan}).

    Analogous to (\ref{Heinz-57}), for the Laplace-Beltrami operator, we let $\mathbf{H}_{\lambda_{1},\lambda_{2}}(\mathbb{B}^{n})$ denote the class of all functions
$f \in C^{2}(\mathbb{B}^{n})$ mapping $\mathbb{B}^{n}$ into $\mathbb{R}$ that satisfy the hyperbolic {\it Heinz nonlinear differential inequality}.
\beqq
0 \leq f(x)\Delta_{h} f(x)\leq \lambda_{1}(|x|)|\nabla^{h} f(x)|^{2}+\lambda_{2}(|x|)|f(x)|^{2},~x\in\mathbb{B}^{n}, \eeqq
where     $\lambda_{1}$ and $\lambda_{2}$ are
 nonnegative continuous functions in $[0,1)$, 
 and $|\nabla^{h} f(x)|^{2}=(1-|x|^{2})^{2}|\nabla f(x)|^{2}$.

 In the following, we show that the Girela-Pel\'aez conjecture holds positively for more general classes of
 functions induced by the M\"obius invariant Laplacian operator (\ref{eq-mo}).

  \begin{thm}\label{thm-1}
Let $p\in[2,\infty)$ and $\omega$ be a majorant. Suppose that $\lambda_{1}$ and $\lambda_{2}$  are
 nonnegative continuous functions in $[0,1)$ 
 and $$\int_{0}^{r}\frac{\rho^{n-1}g_{n}(\rho,r)\lambda_{2}(\rho)}{(1-\rho^{2})^{n}}d\rho<\frac{1}{np},$$
 where $$g_{n}(\rho,r)=\frac{1}{n}\int_{\rho}^{r}\frac{(1-s^{2})^{n-2}}{s^{n-1}}ds.$$
For $r\in[0,1)$, if $f\in\mathbf{H}_{\lambda_{1},\lambda_{2}}(\mathbb{B}^{n})$ such that
\be\label{eq-k-1}M_{p}(r,\nabla f)\leq C\omega\left(\frac{1}{1-r}\right), \ee
then
\beqq
M_{p}(r,f)&\leq&\frac{1}{\sqrt{C(r)}}\left[|f(0)|^{2}+
\frac{2^{n-2}pC\omega(1)}{n-1}\int_{0}^{r}\left(p-1+\lambda_{1}(\rho)\right)
\omega\left(\frac{1}{1-\rho}\right)d\rho\right]^{\frac{1}{2}},
\eeqq where $C(r)=\left(1-np\int_{0}^{r}\frac{\rho^{n-1}g_{n}(\rho,r)\lambda_{2}(\rho)}{(1-\rho^{2})^{n}}d\rho\right).$
\end{thm}

If we set $\lambda_1\equiv M$, where $M$ is a nonnegative constant and $\omega(t)=t$ for $t \geq 0$ in Theorem \ref{thm-1}, we obtain the following result.

\begin{cor}
Let $p\in[2,\infty)$ be a constant. Suppose that $\lambda_1\equiv M$, where $M$ is a nonnegative constant,
and $\lambda_{2}$ is
 a nonnegative continuous functions in $[0,1)$
such that $$\sup_{r\in [0,1)}\int_{0}^{r}\frac{\rho^{n-1}g_{n}(\rho,r)\lambda_{2}(\rho)}{(1-\rho^{2})^{n}}d\rho<\frac{1}{np}.$$

 For $r\in[0,1)$, if $f\in\mathbf{H}_{\lambda_{1},\lambda_{2}}(\mathbb{B}^{n})$ such that
$$M_{p}(r,\nabla f)=O \left ( \frac{1}{1-r} \right  )  ~\mbox{ as $r\rightarrow1^{-}$},
$$
then
$$ M_{p}(r,f)=O \left (\left(\log\frac{1}{1-r}\right)^{\frac{1}{2}} \right ) ~\mbox{ as $r\rightarrow 1^{-}$}.
$$
\end{cor}


Let \( \lambda : \mathbb{D} \to [0, \infty) \) be continuous and \( f = u+iv\in C^2(\mathbb{D}) \). The elliptic partial differential equation (or briefly the PDE) in the form
\be\label{Yu-1}
\Delta f(z) = \lambda(z)f(z),~z=x+iy\in\mathbb{D},
\ee
is called the {\it non-homogeneous Yukawa PDE}
(see \cite{CPR, CRW}). If \( \lambda \) in (\ref{Yu-1}) is a positive constant function, then we have the usual Yukawa PDE, which first arose from the work of the Japanese Nobel physicist Hideki Yukawa. He used this equation to describe the nuclear potential of a point charge as \( e^{-\sqrt{\lambda}r}/r \) (cf. \cite{Ar,BS,CPR}).
Replacing $\Delta$ by $\Delta_{h}$ in \eqref{Yu-1}, we consider the \textit{non-homogeneous hyperbolic Yukawa equation} in $\mathbb{B}^{n}$ as follows:
\be\label{Yu-2}
\Delta_{h}f(x)=\lambda(x)f(x),~ x\in\mathbb{B}^{n},
\ee
where $n \geq 2$, $f \in C^{2}(\mathbb{B}^{n})$, and $\lambda : \mathbb{B}^{n} \to [0, \infty)$ is a continuous function.

By Theorem \ref{thm-1}, we obtain the following result for the solutions to \eqref{Yu-2}
under the condition \eqref{Yukawa-condition}.

\begin{cor}\label{cor-Yu-1}
Let $p\in[2,\infty)$ and $\omega$ be a majorant. Suppose that $\lambda$ is a
 nonnegative continuous functions in $\mathbb{B}^{n}$ with
 \begin{align}\label{Yukawa-condition} 
 \int_{0}^{r}\frac{\rho^{n-1}g_{n}(\rho,r)\lambda_2(\rho)}{(1-\rho^{2})^{n}}d\rho&<\frac{1}{np},
\end{align} 
where $\lambda_2$ is a nonnegative continuous functions in $[0,1)$ defined by
\begin{align*}
\lambda_2(t)=\sup_{|x|=t}\lambda(x).
\end{align*}
For $r\in[0,1)$, if $f$ satisfies (\ref{Yu-2}) such that
\beqq
M_{p}(r,\nabla f)\leq C\omega\left(\frac{1}{1-r}\right),
\eeqq
then
\beqq
M_{p}(r,f)&\leq&\frac{1}{\sqrt{C_{1}(r)}}\left[|f(0)|^{2}+
\frac{2^{n-2}pC\omega(1)(p-1)}{n-1}\int_{0}^{r}
\omega\left(\frac{1}{1-\rho}\right)d\rho\right]^{\frac{1}{2}},
\eeqq
where $C_{1}(r)=\left(1-np\int_{0}^{r}\frac{\rho^{n-1}g_{n}(\rho,r)\lambda_2(\rho)}{(1-\rho^{2})^{n}}d\rho\right).$
\end{cor}

In particular, if $\lambda\equiv0$ and $\omega(t)=t$ for $t \geq 0$ in Corollary \ref{cor-Yu-1}, then we have the following result.

\begin{cor}
Let $p\in[2,\infty)$ and let $f$ be a hyperbolic harmonic function in $\mathbb{B}^{n}$.
For $r\in[0,1)$, if
$$M_{p}(r,\nabla f)=O \left ( \frac{1}{1-r} \right  )  ~\mbox{ as $r\rightarrow1^{-}$},
$$
then
$$ M_{p}(r,f)=O \left (\left(\log\frac{1}{1-r}\right)^{\frac{1}{2}} \right ) ~\mbox{ as $r\rightarrow 1^{-}$}.
$$
\end{cor}




The remainder of this paper is organized as follows.
The proofs of Theorems \ref{thm-1.1} and \ref{thm-1.2} are presented in Section \ref{Sec-2}, and Theorem \ref{thm-1} is proved in Section \ref{Sec-3}.

	\section{Hardy-Littlewood type phenomenon for the M\"obius invariant Laplacian operator}\label{Sec-2}


\subsection*{Proof of Theorem \ref{thm-1.1}}
This section is devoted to the proof of Theorem \ref{thm-1.1}. Given its complexity, we break the proof down into several lemmas.

\begin{Lem}\label{Lem-A}\rm{(}\cite[p. 19]{Z59}\rm{)}
	 For  $\nu \geqslant 1$, Minkowski's inequality in its infinite form is given by
	$$
	\left(\int_{A_{1}}\left|\int_{B_{1}} \mathcal{X}(\zeta, \xi) d \mu_{\xi}\right|^{\nu} d \mu_{\zeta}\right)^{\frac{1}{\nu}} \leqslant \int_{B_{1}}\left(\int_{A_{1}}|\mathcal{X}(\zeta, \xi)|^{\nu} d \mu_{\zeta}\right)^{\frac{1}{\nu}} d \mu_{\xi},
	$$
	where  $A_{1}$  and  $B_{1}$  are measurable sets with positive measures $ d \mu_{\zeta} $ and $ d \mu_{\xi}$, respectively, and  $\mathcal{X}$  is integrable on $ A_{1} \times B_{1}$.
\end{Lem}

Based on Lemma \ref{Lem-A}, we obtain the following results.

	\begin{lem}\label{lem-1.1}
		Let  $n \geq 2$, $\alpha \in(-1, \infty)$, $p\in [1,\infty)$  and  $\omega$  be a majorant. Suppose that \eqref{eq-CCHL-02} holds. Then there is a positive constant  $C$  such that
		\begin{align*}
		\left(\int_{SO(n)}\left|P_{\alpha}[\varphi](Rr \xi)-\varphi(R\xi)\right|^{p}d\mu(R)\right)^{\frac{1}{p}}&\leq C \|\varphi\|_{\Lambda_{\omega, p}(\mathbb{S}^{n-1})}\omega(1-r)
		\end{align*}
		for all $\varphi \in \Lambda_{\omega, p}(\mathbb{S}^{n-1})\cap C ( \mathbb{S}^{n-1})$, $ \xi \in \mathbb{S}^{n-1}$  and  $r \in[0,1)$.
	\end{lem}
	\begin{proof}
%
	Let  $\zeta=\left(\zeta_{1}, \ldots, \zeta_{n}\right) \in \mathbb{S}^{n-1}$  and  $x=r \xi=\left(x_{1}, \ldots, x_{n}\right)$, where  $\xi \in \mathbb{S}^{n-1}$.
	By the Minkowski inequality, we have
	\begin{equation}\label{eq-1.3}
	\left(\int_{SO(n)}\left|P_{\alpha}[\varphi](Rr \xi)-\varphi(R\xi)\right|^{p}d\mu(R)\right)^{\frac{1}{p}} \leq I_{1}+I_{2},
\end{equation}
	where
	\begin{align}\label{eq-1.4}
		I_{1}&=\left(\int_{SO(n)}\left(\int_{\mathbb{S}^{n-1}}\left|P_{\alpha}(x, \zeta)\right||\varphi(R\zeta)-\varphi(R\xi)| d \sigma(\zeta)\right)^{p}d\mu(R)\right)^{\frac{1}{p}}
	\end{align}
	and
	\begin{align*}
	I_{2}&=\left(\int_{SO(n)}|\varphi(R\xi)|^{p}\left|P_{\alpha}[1](x)-1\right|^{p}d\mu(R)\right)^{\frac{1}{p}}.	
\end{align*}
	
	Note that
	$$
	\int_{SO(n)}|\varphi(R\zeta)-\varphi(R\xi)|^{p} d\mu(R) \leq\|\varphi\|^{p}_{\Lambda_{\omega, p}(\mathbb{S}^{n-1}),s} \omega(|\zeta-\xi|)^{p}
	$$
	and
	$$
	|\zeta-\xi| \leq|\xi-x|+|x-\zeta|=1-|x|+|x-\zeta| \leq 2|x-\zeta|,
	$$
	which imply that
	\begin{align}\label{eq-new-1.5}
	\int_{SO(n)}|\varphi(R\zeta)-\varphi(R\xi)|^{p} d\mu(R)
	&\leq
	\|\varphi\|^{p}_{\Lambda_{\omega, p}(\mathbb{S}^{n-1}),s} \omega(2|x-\zeta|)^{p}
	\\ \nonumber
	&\leq
	2^{p}\|\varphi\|^{p}_{\Lambda_{\omega, p}(\mathbb{S}^{n-1}),s} \omega(|x-\zeta|)^{p}.
	\end{align}
	
	It follows from Lemma \ref{Lem-A}, \eqref{eq-1.4} and \eqref{eq-new-1.5} that
	\begin{align*}
	I_{1}  &\leq \int_{\mathbb{S}^{n-1}}\left(\int_{SO(n)}\left|P_{\alpha}(x, \zeta)\right|^{p}|\varphi(R\zeta)-\varphi(R\xi)|^{p} d\mu(R)\right)^{\frac{1}{p}}d\sigma(\zeta)\\
 &= \int_{\mathbb{S}^{n-1}}\left|P_{\alpha}(x, \zeta)\right|\left(\int_{SO(n)}|\varphi(R\zeta)-\varphi(R\xi)|^{p} d\mu(R)\right)^{\frac{1}{p}}d\sigma(\zeta)\\
	&\leq2 C_{n, \alpha}\|\varphi\|_{\Lambda_{\omega, p}(\mathbb{S}^{n-1}),s}\left(1-|x|^{2}\right)^{1+\alpha} J_{1}(x),
\end{align*}
	where
	\begin{align}\label{eq-J1}
	J_{1}(x)&=\int_{\mathbb{S}^{n-1}} \frac{\omega(|x-\zeta|)}{|x-\zeta|^{n+\alpha}} d \sigma(\zeta).
	\end{align}
	Since $\omega$ satisfies \eqref{eq-CCHL-02}, by \cite[eq. (3.11)]{chen24}, we know that
	\begin{align}\label{eq-J1-estimate}
	J_{1}(x)\leq C\frac{\omega(1-r)}{(1-r)^{1+\alpha}},
	\quad |x|=r\in [0,1).
	\end{align}
	Therefore, there is a positive constant $C$ such that
	\begin{equation}\label{eq-1.6}
		I_{1}\leq C \|\varphi\|_{\Lambda_{\omega, p}(\mathbb{S}^{n-1}),s}\omega(1-r),
		\quad r\in [0,1).
	\end{equation}
	
	Next, we estimate $I_{2}$. Because $\varphi\in C(\mathbb{S}^{n-1})$, so
	$$
	I_{2}\leq \sup_{\zeta \in \mathbb{S}^{n-1}}|\varphi(\zeta)| \left|P_{\alpha}[1](x)-1\right|\operatorname{vol}(SO(n))^{\frac{1}{p}}.
	$$
	Since $\omega$ satisfies \eqref{eq-CCHL-02}, by \cite[Lemma 3.3]{chen24}, we know that
     \begin{equation}\label{eq-1.7}
     	I_{2}\leq C \sup_{\zeta \in \mathbb{S}^{n-1}}|\varphi(\zeta)| \omega(1-r)\operatorname{vol}(SO(n))^{\frac{1}{p}}.
     \end{equation}
	Thus, combining \eqref{eq-1.3}, \eqref{eq-1.6} and \eqref{eq-1.7}  gives that there is a positive constant
	$C> 0$ such that
$$\left(\int_{SO(n)}\left|P_{\alpha}[\varphi](Rr \xi)-\varphi(R\xi)\right|^{p}d\mu(R)\right)^{\frac{1}{p}}\leq C\|\varphi\|_{\Lambda_{\omega, p}(\mathbb{S}^{n-1})} \omega(1-r).$$
This completes the proof.
\end{proof}
	
	We will make use of the following formula.
For $j\in\{1,\cdots,n\}$, $\alpha\in (-1,\infty)$,
$x\in \mathbb{B}^n$ and
$\zeta \in \mathbb{S}^{n-1}$,
we have
 \begin{align*}
\frac{\partial}{\partial x_{j}}{P}_{\alpha}(x,\zeta) &= -2C_{n, \alpha}(1+\alpha)\frac{(1-|x|^{2})^{\alpha}x_j}
{|x-\zeta|^{n+\alpha}} -C_{n, \alpha}(n+\alpha)\frac{(1-|x|^{2})^{1+\alpha}(x_{j}-\zeta_{j})}{|x-\zeta|^{n+\alpha+2}}.
 \end{align*}
Then for any $R\in SO(n)$,
we obtain that
\begin{align}
\label{eq-norm-partial}
\left|\frac{\partial}{\partial x_{j}}{P}_{\alpha}(Rx,R\zeta)\right| &\leq
 2C_{n, \alpha}(1+\alpha)\frac{(1-|x|^{2})^{\alpha}}
{|x-\zeta|^{n+\alpha}} +C_{n, \alpha}(n+\alpha)\frac{(1-|x|^{2})^{1+\alpha}}{|x-\zeta|^{n+\alpha+1}}.
 \end{align}


\begin{lem}\label{thm-1b} Let  $n \geq 2$, $\alpha \in(-1, \infty)$, $p\in [1,\infty)$  and  $\omega$  be a majorant.
Suppose that \eqref{eq-CCHL-02} holds.
Then
there is a positive constant $C$
such that
\[
\left(\int_{SO(n)}|\nabla P_{\alpha}[\varphi](Rx)|^pd\mu(R)\right)^{\frac{1}{p}}\leq C\|\varphi\|_{\Lambda_{\omega, p}(\mathbb{S}^{n-1})}\frac{\omega(1-|x|)}{1-|x|}
\]
for all $\varphi \in \Lambda_{\omega, p}(\mathbb{S}^{n-1})\cap C ( \mathbb{S}^{n-1})$ and $x\in \mathbb{B}^n$,
where 
$
\nabla P_{\alpha}[\varphi](Rx)=\nabla P_{\alpha}[\varphi](y)|_{y=Rx}
$.

\end{lem}

\begin{proof}
	Let $\zeta=(\zeta_{1},\ldots,\zeta_{n})\in\mathbb{S}^{n-1}$ and $x=r\xi=(x_{1},\ldots,x_{n})$, where $\xi\in\mathbb{S}^{n-1}$.
 For $j\in\{1,\cdots,n\}$, by the Minkowski inequality, we have
\begin{equation}
\label{I1I2-b}
\left(\int_{SO(n)}\left|\frac{\partial}{\partial x_{j}}P_{\alpha}[\varphi](Rx)\right|^{p}d\mu(R)\right)^{\frac{1}{p}}\leq L_1+L_2,
\end{equation}
where
\[
L_1=\left(\int_{SO(n)}\left(\int_{\mathbb{S}^{n-1}}\left|\frac{\partial}{\partial x_{j}}{P}_{\alpha}(Rx,\zeta)\right| |\varphi(\zeta)-\varphi(R\xi)|d\sigma(\zeta)\right)^{p}d\mu(R)\right)^{\frac{1}{p}}
\]
and
\[
L_2=\left(\int_{SO(n)}|\varphi(R\xi)|^{p}\left|\frac{\partial}{\partial x_{j}}P_{\alpha}[1](Rx)\right|^{p}d\mu(R)\right)^{\frac{1}{p}}.
\]

From the rotation invariance of $d\sigma(\zeta)$ and (\ref{eq-norm-partial}),  we see that,
\begin{align}\label{eq-r-1-0-b}
L_1&=\left(\int_{SO(n)}\left(\int_{\mathbb{S}^{n-1}}\left|\frac{\partial}{\partial x_{j}}{P}_{\alpha}(Rx,R\zeta)\right||\varphi(R\zeta)-\varphi(R\xi)|d\sigma(\zeta)\right)^{p}d\mu(R)\right)^{\frac{1}{p}} \\ \nonumber
&\leq  L_{1,1}+L_{1,2},
\end{align}
where
\begin{align*}
L_{1,1}
&= 2C_{n, \alpha}(1+\alpha)\left(\int_{SO(n)}\left(\int_{\mathbb{S}^{n-1}}\frac{|\varphi(R\zeta)-\varphi(R\xi)|(1-|x|^{2})^{\alpha}}{|x-\zeta|^{n+\alpha}}d\sigma(\zeta)\right)^{p}d\mu(R)\right)^{\frac{1}{p}}
\end{align*}
and
\begin{align*}
L_{1,2}
&=
C_{n, \alpha}(n+\alpha)\left(\int_{SO(n)}\left(\int_{\mathbb{S}^{n-1}}\frac{|\varphi(R\zeta)-\varphi(R\xi)|(1-|x|^{2})^{1+\alpha}}{|x-\zeta|^{n+\alpha+1}}d\sigma(\zeta)\right)^{p}d\mu(R)\right)^{\frac{1}{p}}.
\end{align*}

It follows from Lemma \ref{Lem-A}, (\ref{eq-new-1.5}) and (\ref{eq-r-1-0-b})  that

\beq\label{I1-2-b}
L_1
&\leq&4C_{n, \alpha}(1+\alpha)\|\varphi\|_{\Lambda_{\omega, p}(\mathbb{S}^{n-1}),s}(1-|x|^{2})^{\alpha}
J_{1}\\ \nonumber
&&+2C_{n, \alpha}(n+\alpha)\|\varphi\|_{\Lambda_{\omega, p}(\mathbb{S}^{n-1}),s}(1-|x|^{2})^{1+\alpha}
J_{2},
\eeq where $J_{1}$ is defined in (\ref{eq-J1}) and
$$J_{2}=\int_{\mathbb{S}^{n-1}}\frac{\omega(|x-\zeta|)}{|x-\zeta|^{n+\alpha+1}}d\sigma(\zeta).$$
%
Since $\omega$ satisfies \eqref{eq-CCHL-02}, by \cite[eq. (3.20)]{chen24}, we know that there is a positive constant $C$ such that

\be\label{eq-rt-5-b}
J_{2}\leq
 C \frac{\omega(1-r)}{(1-r)^{2+\alpha}}.
 \ee
By (\ref{eq-J1-estimate}), (\ref{I1-2-b}) and (\ref{eq-rt-5-b}), we conclude that there is a positive constant $C$ such that
\be\label{F-01-b} L_{1}\leq C \|\varphi\|_{\Lambda_{\omega, p}(\mathbb{S}^{n-1}),s}\frac{\omega(1-r)}{1-r}.\ee

Next, we estimate $L_2$.
Since $P_{0}[1](x)\equiv 1$,
it suffices to consider the case $\alpha>-1$ and $\alpha \neq 0$.
Because $\varphi\in C(\mathbb{S}^{n-1})$, so
	$$
	L_{2}
\leq \operatorname{vol}(SO(n))^{\frac{1}{p}}
\sup_{\zeta \in \mathbb{S}^{n-1}}|\varphi(\zeta)|
\sup_{R \in SO(n)}\left|\frac{\partial}{\partial x_{j}}P_{\alpha}[1](Rx)\right|.
	$$
	Since $\omega$ satisfies \eqref{eq-CCHL-02}, by \cite[eq. (3.22) and eq. (3.23)]{chen24}, we know that
     \begin{equation}\label{eq--new-1.7}
     	L_{2}\leq C \sup_{\zeta \in \mathbb{S}^{n-1}}|\varphi(\zeta)|\sup_{R \in SO(n)}\frac{\omega(1-|Rx|)}{1-|Rx|}\operatorname{vol}(SO(n))^{\frac{1}{p}}=C \sup_{\zeta \in \mathbb{S}^{n-1}}|\varphi(\zeta)|\frac{\omega(1-r)}{1-r}\operatorname{vol}(SO(n))^{\frac{1}{p}}.
     \end{equation}

Thus, combining
 (\ref{I1I2-b}), (\ref{F-01-b}) and  (\ref{eq--new-1.7})  gives that
there is a positive constant $C $ such that

\[
\left(\int_{SO(n)}|\nabla P_{\alpha}[\varphi](Rx)|^{p}d\mu(R)\right)^{\frac{1}{p}}\leq C\|\varphi\|_{\Lambda_{\omega, p}(\mathbb{S}^{n-1})}\frac{\omega(1-|x|)}{1-|x|}.
\]
The proof of this lemma is finished.
\end{proof}

	\begin{lem}\label{lem-1.2}
		Let  $n \geq 2$, $\alpha \in(-1, \infty)$, $p\in [1,\infty)$  and  $\omega$  be a majorant. Suppose that \eqref{eq-CCHL-02} holds. Then there is a positive constant $ C $ such that
		$$
		\left(\int_{SO(n)}\left|P_{\alpha}[\varphi]\left(Rr_{1} \xi\right)-P_{\alpha}[\varphi]\left(Rr_{2} \xi\right)\right|^{p}d\mu(R)\right)^{\frac{1}{p}} \leq C \|\varphi\|_{\Lambda_{\omega, p}(\mathbb{S}^{n-1})}\omega\left(r_{1}-r_{2}\right)
		$$
		for all $\varphi \in \Lambda_{\omega, p}(\mathbb{S}^{n-1})\cap C ( \mathbb{S}^{n-1})$, $\xi \in \mathbb{S}^{n-1}$  and  $r_{1}$, $r_{2}$  with  $0 \leq r_{2}<r_{1}<1$.
	\end{lem}
	\begin{proof}
	First, assume that $ 1-r_{1} \leq r_{1}-r_{2}$. Then, we have
	$$
	1-r_{2}=\left(1-r_{1}\right)+\left(r_{1}-r_{2}\right) \leq 2\left(r_{1}-r_{2}\right) .
	$$
	Therefore, by the Minkowski inequality and Lemma \ref{lem-1.1}, there is a positive constant  $C$  such that
	\begin{align*}
		\mathcal{L}_{p}[P_{\alpha}[\varphi]]\left(r_{1} \xi, r_{2} \xi\right)
		 &\leq\left(\int_{SO(n)}\left|P_{\alpha}[\varphi](Rr_{1} \xi)-\varphi(R\xi)\right|^{p}d\mu(R)\right)^{\frac{1}{p}}\\
		&\quad +\left(\int_{SO(n)}\left|P_{\alpha}[\varphi](Rr_{2} \xi)-\varphi(R\xi)\right|^{p}d\mu(R)\right)^{\frac{1}{p}} \\
		 &\leq C\|\varphi\|_{\Lambda_{\omega, p}(\mathbb{S}^{n-1})}\left(\omega\left(1-r_{1}\right)+\omega\left(1-r_{2}\right)\right) \\
		 &\leq C\|\varphi\|_{\Lambda_{\omega, p}(\mathbb{S}^{n-1})}\left(\omega\left(r_{1}-r_{2}\right)+\omega\left(2\left(r_{1}-r_{2}\right)\right)\right. \\
		 &\leq 3 C \|\varphi\|_{\Lambda_{\omega, p}(\mathbb{S}^{n-1})}\omega\left(r_{1}-r_{2}\right).
	\end{align*}
	
	Next, we consider the case  $1-r_{1}>r_{1}-r_{2}$. Since  $$\frac{1-r_{1}}{r_{1}-r_{2}}>1,$$ by using Lemmas \ref{Lem-A} and \ref{thm-1b}, there is a positive constant  $C$  such that
	
	\begin{align*}
		\mathcal{L}_{p}[P_{\alpha}[\varphi]]\left(r_{1} \xi, r_{2} \xi\right)
		& \leq \left(\int_{SO(n)}\left(\int_{r_{2}}^{r_{1}}\left|\frac{d}{d t} P_{\alpha}[\varphi](Rt \xi)\right| d t \right)^{p}d\mu(R)\right)^{\frac{1}{p}}\\
		& \leq \left(\int_{SO(n)}\left(\int_{r_{2}}^{r_{1}}\left|\nabla P_{\alpha}[\varphi](Rt\xi) R\xi\right| d t\right)^{p}d\mu(R)\right)^{\frac{1}{p}} \\
		& \leq \int_{r_{2}}^{r_{1}}\left(\int_{SO(n)}\left|\nabla P_{\alpha}[\varphi](Rt\xi) R\xi\right|^{p}d\mu(R) \right)^{\frac{1}{p}}dt\\
		& \leq C\|\varphi\|_{\Lambda_{\omega, p}(\mathbb{S}^{n-1})} \int_{r_{2}}^{r_{1}} \frac{\omega(1-t)}{1-t} d t \\
		& \leq C\|\varphi\|_{\Lambda_{\omega, p}(\mathbb{S}^{n-1})} \frac{r_{1}-r_{2}}{1-r_{1}} \omega\left(1-r_{1}\right) \\
		& \leq C \|\varphi\|_{\Lambda_{\omega, p}(\mathbb{S}^{n-1})}\omega\left(r_{1}-r_{2}\right).
	\end{align*}
	This completes the proof.
	\end{proof}
	
	\begin{lem}\label{lem-1.3}
		Let  $n \geq 2$, $\alpha \in(-1, \infty)$, $p\in [1,\infty)$  and  $\omega$  be a majorant. Suppose that \eqref{eq-CCHL-02} holds. Then there is a positive constant $ C$  such that
		$$
		\left(\int_{SO(n)}\left|P_{\alpha}[\varphi](Rx)-P_{\alpha}[\varphi](Ry)\right|^{p}d\mu(R) \right)^{\frac{1}{p}}
		\leq C\|\varphi\|_{\Lambda_{\omega, p}(\mathbb{S}^{n-1})} \omega(|x-y|)
		$$
		for all $\varphi \in \Lambda_{\omega, p}(\mathbb{S}^{n-1})\cap C ( \mathbb{S}^{n-1})$ and $x, y \in \mathbb{B}^{n}$  with  $|x|=|y|$.
	\end{lem}
	\begin{proof}

We may assume that $|x|=|y|>0$.		
		First, assume that $ 1-|x| \leq|x-y|$. Let $ \xi=x /|x|$  and  $\zeta=y /|y|$.
  Then, combining Lemma \ref{lem-1.1} and  $$|\xi-\zeta| \leq 3|x-y|,$$ we deduce that there is a positive constant  $C$  such that
		\begin{align*}
			\mathcal{L}_{p}[P_{\alpha}[\varphi]]\left(x, y\right)& \leq\left(\int_{SO(n)}\left|P_{\alpha}[\varphi](Rx)-\varphi(R\xi)\right|^{p}d\mu (R)\right)^{\frac{1}{p}}\\ &\quad+\left(\int_{SO(n)}\left|\varphi(R\xi)-\varphi(R\zeta)\right|^{p}d\mu(R) \right)^{\frac{1}{p}}\\
			&\quad +\left(\int_{SO(n)}\left|\varphi(R\zeta)-P_{\alpha}[\varphi](Ry)\right|^{p}d\mu(R) \right)^{\frac{1}{p}}\\
			& \leq C\|\varphi\|_{\Lambda_{\omega, p}(\mathbb{S}^{n-1})}(\omega(1-|x|)+\omega(|\xi-\zeta|)+\omega(1-|y|)) \\
			& \leq C\|\varphi\|_{\Lambda_{\omega, p}(\mathbb{S}^{n-1})}(2 \omega(|x-y|)+\omega(3|x-y|)) \\
			& \leq 5 C \|\varphi\|_{\Lambda_{\omega, p}(\mathbb{S}^{n-1})}\omega(|x-y|) .
		\end{align*}
	
		Next, we consider the case $ 1-|x|>|x-y|$. Since  $\frac{1-|x|}{|x-y|}>1$, by using Lemmas \ref{Lem-A} and \ref{thm-1b}, there is a positive constant $C$  such that
		\begin{align*}
			\mathcal{L}_{p}[P_{\alpha}[\varphi]]\left(x, y\right)
			 & \leq \left(\int_{SO(n)}\left(\int_{0}^{1}\left|\frac{d}{d t} P_{\alpha}[\varphi](t Rx+(1-t)Ry)\right| d t\right)^{p}d\mu(R) \right)^{\frac{1}{p}} \\
			& \leq \left(\int_{SO(n)}\left(\int_{0}^{1}\left|\nabla P_{\alpha}[\varphi](t Rx+(1-t) Ry)(Rx-Ry)\right| d t \right)^{p}d\mu(R) \right)^{\frac{1}{p}}\\
			&\leq \int_{0}^{1}\left( \int_{SO(n)}\left|\nabla P_{\alpha}[\varphi](t Rx+(1-t) Ry)(Rx-Ry)\right|^{p}d\mu(R)\right)^{\frac{1}{p}}d t\\
			& \leq C\|\varphi\|_{\Lambda_{\omega, p}(\mathbb{S}^{n-1})} \frac{|x-y|}{1-|x|} \omega(1-|x|) \\
			& \leq C\|\varphi\|_{\Lambda_{\omega, p}(\mathbb{S}^{n-1})} \omega(|x-y|).
		\end{align*}
		This completes the proof.
	\end{proof}

 \begin{lem}\label{lem-CCHL-1}
    	Let  $n \geq 2$, $\alpha \in(-1, \infty)$,   $\omega$  be a majorant  and $p \in[1, \infty]$  be a constant.
    	\begin{enumerate}
    	\item[{\rm $\left(\mathcal{D}_{1}\right)$}]
    	  If  $\varphi \in \Lambda_{\omega, p}(\mathbb{S}^{n-1})$  and $\varphi$  is continuous on  $\mathbb{S}^{n-1}$, then  $P_{\alpha}[\varphi] \in \Lambda_{\omega, p}(\overline{\mathbb{B}^{n}})$.
    	\item[{\rm $\left(\mathcal{D}_{2}\right)$}]  There exists a positive constant $C$ such that inequality (\ref{eq-CCHL-02}) holds for all $\delta \in (0, \pi]$.
    	\end{enumerate}
    Then,  $\left(\mathcal{D}_{2}\right) \Rightarrow\left(\mathcal{D}_{1}\right)$  for  $p \in[1, \infty]$, and
    	$\left(\mathcal{D}_{2}\right) \Leftrightarrow\left(\mathcal{D}_{1}\right)$  for  $p=\infty$.
    \end{lem}


	\begin{proof}
 If $p=\infty$, the result follows directly from Theorem \ref{Thm-D}.
 In the following, we  consider the case when $p\in[1,\infty)$.
 Assume that $\left(\mathcal{D}_{2}\right)$ holds.
 Let $\varphi \in \Lambda_{\omega, p}(\mathbb{S}^{n-1})\cap C ( \mathbb{S}^{n-1})$.
 We will show that there is a positive constant $C$ such that
 $$
	\mathcal{L}_{p}[P_{\alpha}[\varphi]](x,y) \leqslant C \|\varphi\|_{\Lambda_{\omega, p}(\mathbb{S}^{n-1})} \omega\left(\left|x-y\right|\right), \quad x, y\in \mathbb{B}^{n}.
	$$
We may assume that $ |x| \geq|y|$  with  $x \neq y$. If  $y=t x$  for some  $t \in[0,1]$, then the result follows from Lemma \ref{lem-1.2}. So, it suffices to consider the case  $y$ is not contained in the segment between $0$ and $ x$. Let  $z=\frac{|x|}{|y|} y$. Then
		$$|y-z| \leq|x-y| \text { and }|x-z| \leq 2|x-y|,$$
		which, together with Lemmas \ref{lem-1.2}, \ref{lem-1.3} and the Minkowski inequality, implies that there is a positive constant  $C$  such that
		\begin{align*}
			\left(\int_{SO(n)}\left|P_{\alpha}[\varphi](Rx)-P_{\alpha}[\varphi](Ry)\right|^{p}d\mu \right)^{\frac{1}{p}}
 \leq &\left(\int_{SO(n)}\left|P_{\alpha}[\varphi](Rx)-P_{\alpha}[\varphi](Rz)\right|^{p}d\mu \right)^{\frac{1}{p}}\\
			&+\left(\int_{SO(n)}\left|P_{\alpha}[\varphi](Rz)-P_{\alpha}[\varphi](Ry)\right|^{p}d\mu \right)^{\frac{1}{p}}\\
			\leq & C\|\varphi\|_{\Lambda_{\omega, p}(\mathbb{S}^{n-1})} (\omega(|x-z|)+\omega(|z-y|)) \\
			\leq& 3 C\|\varphi\|_{\Lambda_{\omega, p}(\mathbb{S}^{n-1})}  \omega(|x-y|) .
		\end{align*}
The proof of this lemma is complete.		
\end{proof}

Next, we consider the equivalence of \eqref{eq-CCHL-02} and \eqref{eq-CCHL-03}.
Assume that \eqref{eq-CCHL-02} holds.
Then for all $c\in (0, \pi)$ and $0<\lambda\leq c$,
we have
$$
\int_{c}^{\pi}\frac{\omega(t)\sin^{n-2}t}{t^{n+\alpha}}\,dt
\leq
\int_{\lambda}^{\pi}\frac{\omega(t)\sin^{n-2}t}{t^{n+\alpha}}\,dt\leq\, C
\frac{\omega(\lambda)}{\lambda^{1+\alpha}},
$$
which implies that
\be
\label{eq-inf}
\inf_{\lambda\in (0,c]}\frac{\omega(\lambda)}{\lambda^{1+\alpha}}>0.
\ee
Further, we have the following result.
\begin{lem}
\label{lem-necessary}
Let $n\geq2$, $\alpha\in(-1,\infty)$ and $\omega$ be a majorant.
For each $\varepsilon\in (0, \pi]$,
there is a positive constant $C_{\varepsilon,\alpha}$ such that
\begin{align*}
\delta^{1+\alpha}\int_{\delta}^{\pi}\frac{\omega(t)}{t^{2+\alpha}}\,dt&\leq\, C_{\varepsilon,\alpha}
\omega(\delta)
\end{align*}
for  all $\delta\in [\varepsilon, \pi]$.
\end{lem}

\begin{proof}
We may assume that $\varepsilon\in (0,\pi)$.
Since $\omega(t)/t$ is non-increasing, we  have
\begin{align*}
\delta^{1+\alpha}\int_{\delta}^{\pi}\frac{\omega(t)}{t^{2+\alpha}}\,dt\leq
\delta^{\alpha}\omega(\delta)\int_{\delta}^{\pi}\frac{1}{t^{1+\alpha}}\,dt.
\end{align*}

If $\alpha>0$, then we have
\begin{align*}
\delta^{\alpha}\omega(\delta)\int_{\delta}^{\pi}\frac{1}{t^{1+\alpha}}\,dt
&\leq \delta^{\alpha}\omega(\delta)\frac{1}{\alpha}\delta^{-\alpha}
 =\frac{1}{\alpha}\omega(\delta).
\end{align*}

If $\alpha=0$, then we have
\begin{align*}
\delta^{\alpha}\omega(\delta)\int_{\delta}^{\pi}\frac{1}{t^{1+\alpha}}\,dt
&=
\omega(\delta)(\log \pi-\log \delta)
 \leq
(\log \pi-\log \varepsilon)\omega(\delta).
\end{align*}

If $\alpha \in (-1,0)$,
then we have
\begin{align*}
\delta^{\alpha}\omega(\delta)\int_{\delta}^{\pi}\frac{1}{t^{1+\alpha}}\,dt
&\leq
\delta^{\alpha}\omega(\delta)\frac{1}{-\alpha}\pi^{-\alpha}
 \leq
\frac{\varepsilon^{\alpha}\pi^{-\alpha}}{-\alpha}\omega(\delta).
\end{align*}
This completes the proof.
\end{proof}

\begin{lem}\label{lem-CCHL-3}
Let $n\geq2$, $\alpha\in(-1,\infty)$ and $\omega$ be a majorant.
Then there is a positive constant $C$ such that \eqref{eq-CCHL-02} holds for  all $\delta\in (0, \pi]$
if and only if
there is a positive constant $C$ such that \eqref{eq-CCHL-03}
holds
for  all $\delta\in (0, \pi]$.
\end{lem}

\begin{proof}
Since $\sin t\leq t$ for $t\in [0, \pi]$,
it is easy to see that \eqref{eq-CCHL-03} implies \eqref{eq-CCHL-02}.

Conversely, assume that \eqref{eq-CCHL-02} holds.
In view of Lemma \ref{lem-necessary},
every majorant $\omega$ satisfies \eqref{eq-CCHL-03}
for $\delta\in [\pi/2, \pi]$.
So, we may assume that $\delta \in (0, \pi/2]$.
Let
\begin{align*}
F(\delta)=\int_{\delta}^{\pi}\frac{\omega(t)\sin^{n-2}t}{t^{n+\alpha}}\,dt,
\quad
\delta \in (0,\pi/2]
\end{align*}
and
\begin{align*}
G(\delta)=\int_{\delta}^{\pi}\frac{\omega(t)}{t^{2+\alpha}}\,dt,
\quad
\delta \in (0,\pi/2].
\end{align*}
By \eqref{eq-inf}, we see that
\begin{align*}
G(\delta)
&\geq
\int_{\delta}^{\pi/2}\frac{\omega(t)}{t^{2+\alpha}}\,dt
 \geq
\inf_{t\in(0,\pi/2]} \frac{\omega(t)}{t^{1+\alpha}}\int_{\delta}^{\pi/2}\frac{1}{t}\,dt
\to \infty,
\quad \mbox{as } \delta\to 0^+.
\end{align*}
Also,
\begin{align*}
F(\delta)
&\geq
\left(\frac{2}{\pi}\right)^{n-2}\int_{\delta}^{\pi/2}\frac{\omega(t)}{t^{2+\alpha}}\,dt
 \geq
\left(\frac{2}{\pi}\right)^{n-2}  \inf_{t\in(0,\pi/2]} \frac{\omega(t)}{t^{1+\alpha}}  \int_{\delta}^{\pi/2}\frac{1}{t}\,dt
\to \infty,
\quad \mbox{as } \delta\to 0^+.
\end{align*}
Therefore, we have
\begin{align*}
\lim_{\delta \to 0^+}\frac{G(\delta)}{F(\delta)}
&=
\lim_{\delta \to 0^+}\frac{G'(\delta)}{F'(\delta)}
 =
\lim_{\delta \to 0^+}\left(\frac{\delta}{\sin \delta}\right)^{n-2}
 =1.
\end{align*}
Taking into account of the fact that
$F$ and $G$ are positive continuous functions on $(0,\pi/2]$,
we see that
\begin{align*}
\sup_{\delta\in (0, \pi/2]}\frac{G(\delta)}{F(\delta)}\in(0,\infty).
\end{align*}
Combing this with \eqref{eq-CCHL-02} gives that
\begin{align*}
\delta^{1+\alpha}\int_{\delta}^{\pi}\frac{\omega(t)}{t^{2+\alpha}}\,dt
\leq M \omega(\delta)
\sup_{\delta\in (0, \pi/2]}\frac{G(\delta)}{F(\delta)}
\end{align*}
for $\delta \in (0, \pi/2]$.
This completes the proof.
\end{proof}

 Theorem \ref{thm-1.1} follows readily from Lemmas \ref{lem-CCHL-1} and \ref{lem-CCHL-3}.
This completes the proof of the theorem.
\qed

		
\subsection*{Proof of Theorem \ref{thm-1.2}.}
In view of the proof of \cite[Theorem 2.2]{chen24},
we have the following lemma.

\begin{Lem}\label{Lem-I} \rm{(}\cite[Theorem 2.2]{chen24}\rm{)}
	Let  $n \geq 2$, $\alpha \in(-1, \infty)$  and  $\omega$  be a majorant. Suppose that \eqref{eq-CCHL-02} holds.
	Then there is a positive constant  $M$  such that
	$$
	\left|\nabla P_{\alpha}[\varphi](x)\right| \leq M \|\varphi\|_{\Lambda_{\omega}(\mathbb{S}^{n-1})}\frac{\omega(1-|x|)}{1-|x|}
	$$
	for all $\varphi \in \Lambda_{\omega}(\mathbb{S}^{n-1})$ and $x \in \mathbb{B}^{n}$.
\end{Lem}

By the same reasoning as in the proof of  \cite[Lemma 2.2]{Ai2010}, we obtain the following result.
\begin{lem}\label{Lem-H}
	Let $\omega$ be a majorant. Then we have the following:
	
	$(i)$  $\omega$  is subadditive, that is, if  $s, t>0$, then  $\omega(s+t) \leqslant \omega(s)+\omega(t)$;
	
	$(ii)$  $\omega$  is doubling, that is, if  $0<s \leqslant t \leqslant 2 s$, then  $\omega(s) \leqslant \omega(t) \leqslant 2 \omega(s)$.
\end{lem}

\begin{lem}\label{lem-3.1}
 Let  $\omega$ be a majorant  and let  $\tau_{a, \omega}$  be as in \eqref{eq-1.9.1}. Then  $\tau_{a, \omega} \in \Lambda_{\omega}(\mathbb{S}^{n-1})$  and there is a positive constant $C$  independent of  $a \in \mathbb{S}^{n-1}$  such that
$$
\left\|\tau_{a, \omega}\right\|_{\Lambda_{\omega}(\mathbb{S}^{n-1})} \leqslant C.
$$
\end{lem}
\begin{proof}
The proof of the lemma is the same to that of \cite[Lemma 2.4]{Ai2010}, so  we omit it here.
\end{proof}

\begin{lem}\label{lem-omega}
Let  $\omega$  be a majorant and let $\alpha>-1$.
Assume that one of the following conditions is satisfied:

      $(i)$  $\left\|P_{\alpha}\right\|_{\omega\to\omega}<\infty$;

      $(ii)$ there is a constant  $C$  such that
      $$
      P_{\alpha} [\tau_{e_1, \omega}](x) \leq C \omega(|x-e_1|) \quad \text { for } x \in \mathbb{B}^{n},
      $$
      where $e_{1}=(1,0, \ldots, 0)$ is the standard basis vector.

      Then $\omega$ satisfies \eqref{eq-CCHL-02}.
\end{lem}

\begin{proof}
First, assume that (i) holds.
Then by \cite[Theorem 2.1]{chen24},
$\omega$ satisfies \eqref{eq-CCHL-02}.

Next, assume that
(ii) holds.
By the proof of \cite[Theorem 2.1]{chen24},
it suffices to prove \eqref{eq-CCHL-02}
in the case $\delta \in (0,1)$.
Let the function  $\varphi(\zeta)=\omega\left(\arccos \zeta_{1}\right)$  be defined on the unit sphere  $\mathbb{S}^{n-1}$, where  $\zeta_{1}$  is the first coordinate of  $\zeta$.
The angle  $\theta_{\zeta}=\arccos \zeta_{1}$  represents the angle between  $\zeta$  and $e_{1}$, so  $\varphi(\zeta)=\omega (\theta_{\zeta} )$.
By the computation in the proof of \cite[Theorem 2.1]{chen24},
we obtain that
there exists a positive constant $C'$ such that
\begin{align}\label{eq-P}
(1-r)^{1+\alpha}\int_{1-r}^{\pi} \frac{\omega(t)\sin^{n-2}t}{t^{n+\alpha}}\,dt
&\leq
C'P_{\alpha}[\varphi](r,0,\ldots,0),
\quad
r\in (0,1).
\end{align}
Since  $\arccos (1-u)\leq \frac{\pi \sqrt{2}}{2}\sqrt{u} $  for  $u \in[0,2]$,
we have
\begin{align*}
\theta_{\zeta}&=\arccos\left(1-\frac{|\zeta-e_1|^2}{2}\right)
\leq \frac{\pi}{2}|\zeta-e_1|,
\end{align*}
which implies that
\begin{align}\label{eq-omega}
\varphi(\zeta)&\leq \omega\left( \frac{\pi}{2}|\zeta-e_1|\right)
\leq
\frac{\pi}{2}\omega\left(|\zeta-e_1|\right)
=
\frac{\pi}{2} \tau_{e_1, \omega}(\zeta).
\end{align}
Combining \eqref{eq-P}, \eqref{eq-omega} and assumption (ii),
we have
\begin{align*}
(1-r)^{1+\alpha}\int_{1-r}^{\pi} \frac{\omega(t)\sin^{n-2}t}{t^{n+\alpha}}\,dt
&\leq
C'P_{\alpha}[\varphi](r,0,\ldots,0)
\\
&\leq
\frac{C'\pi}{2} P_{\alpha}[\tau_{e_1, \omega}](r,0,\ldots,0)
\\
&\leq
\frac{C C'\pi}{2}\omega\left( 1-r \right)
\quad
r\in (0,1),
\end{align*}
which implies that $\omega$ satisfies \eqref{eq-CCHL-02}
for $\delta \in (0,1)$.
This completes the proof.
\end{proof}

We now proceed to the proof of Theorem \ref{thm-1.2}.
First, we show $(i)  \Rightarrow  (ii)$. Suppose  $\left\|P_{\alpha}\right\|_{\omega \rightarrow \omega}<\infty$. Then,
by Lemma \ref{lem-omega}, we know that
$\omega$ satisfies \eqref{eq-CCHL-02}.
Also, Lemma \ref{lem-3.1} gives
$$
\left\|P_{\alpha}[ \tau_{a, \omega}]\right\|_{\Lambda_{\omega}(\mathbb{B}^{n})} \leqslant\left\|P_{\alpha}\right\|_{\omega \rightarrow \omega}\left\|\tau_{a, \omega}\right\|_{\Lambda_{\omega}(\mathbb{S}^{n-1})} \leqslant C\left\|P_{\alpha}\right\|_{\omega \rightarrow \omega}<\infty.
$$
Hence
$$
\left|P_{\alpha}[ \tau_{a, \omega}](x)-P_{\alpha} [\tau_{a, \omega}](y)\right| \leqslant C \omega(|x-y|) \quad \text { for } x, y \in \mathbb{B}^{n}.
$$

Letting  $y \rightarrow a$  gives
$$
P_{\alpha} [\tau_{a, \omega}](x) \leqslant C \omega\left(\left|x-a\right|\right).
$$

Next, we prove $(ii)  \Rightarrow  (i)$. Suppose $(ii)$ holds. Let  $f \in \Lambda_{\omega}(\mathbb{S}^{n-1})$,  and let  $x, y \in \mathbb{B}^{n}$.
Since $$|P_{\alpha}[f](x)|\leq \|f\|_{\Lambda_{\omega}(\mathbb{S}^{n-1})}|P_{\alpha}[1](x)|$$
and $P_{\alpha}[1]$ is bounded in $\mathbb{B}^n$,
it is sufficient to show that
$$
\left|P_{\alpha} [f](x)-P_{\alpha} [f](y)\right| \leqslant C\|f\|_{\Lambda_{\omega}(\mathbb{S}^{n-1})} \omega(|x-y|).
$$
We may assume that $0<|x|\leq |y|$.

Let us consider two cases.

\noindent {\rm $\mathbf{Case~1.}$} $|x-y| \leqslant \frac{1}{2} (1-|x|)$.

For $z \in \mathbb{B}^{n}\left(x, \frac{1}{2} (1-|x|)\right)=\left\{z\in \mathbb{R}^{n}:|z-x|<\frac{1}{2} (1-|x|)\right\}$, we have
\be\label{Vb-1}
\frac{\omega\left(1-|z|\right)}{1-|z|}\leq\frac{\omega\left(1-\left(|x|+\frac{1-|x|}{2}\right)\right)}{1-\left(|x|+\frac{1-|x|}{2}\right)}=\frac{\omega\left(\frac{1-|x|}{2}\right)}{\frac{1}{2}(1-|x|)}\leq \frac{2\omega\left(1-|x|\right)}{1-|x|}.
\ee
Since  $\omega$ satisfies \eqref{eq-CCHL-02}
by Lemma \ref{lem-omega},
it follows from (\ref{Vb-1}) and  Lemma \ref{Lem-I}  that
$$
\left|\nabla P_{\alpha}[f] (z)\right| \leqslant C\|f\|_{\Lambda_{\omega}(\mathbb{S}^{n-1})} \frac{\omega\left(1-|z|\right)}{1-|z|}\leqslant C\|f\|_{\Lambda_{\omega}(\mathbb{S}^{n-1})} \frac{\omega\left(1-|x|\right)}{1-|x|}
$$
for $z \in B\left(x, \frac{1}{2} (1-|x|)\right)$.
Then, by the mean value theorem, Lemma \ref{Lem-H} ($ii$) and the fact that $\omega(t) / t$  is non-increasing, we obtain
\begin{align}\label{eq-case3.1}
\left|P_{\alpha}[ f](x)-P_{\alpha}[ f](y)\right| \leqslant C\|f\|_{\Lambda_{\omega}(\mathbb{S}^{n-1})} \frac{\omega\left(1-|x|\right)}{1-|x|}|x-y| \leqslant C\|f\|_{\Lambda_{\omega}(\mathbb{S}^{n-1})} \omega(|x-y|).
\end{align}
 \noindent {\rm $\mathbf{Case~2.}$} $|x-y|>\frac{1}{2} (1-|x|)$.

 By assumption, we have  $$|x-y|>\frac{1}{2} (1-|x|) \geqslant \frac{1}{2} (1-|y|).$$
 Let 
$x^*=x/|x|$ and $y^*=y/|y|$.
Then
$$
\left|x^{*}-y^{*}\right| \leqslant|x-y|+|x-x^{*}|+|y-y^{*}| \leqslant 5|x-y|.
$$

Set $g(x)=:P_{\alpha}[1](x)$ in $\mathbb{B}^{n}$.
By $(ii)$ and the facts that $f \in \Lambda_{\omega}(\mathbb{S}^{n-1})$ and $g$ is bounded in $\mathbb{B}^n$, we obtain
\begin{align*}
	\left|P_{\alpha} [f](x)-g(x)f\left(x^{*}\right)\right| & =\left|P_{\alpha} [f_{0}](x)\right| \leqslant \|f\|_{\Lambda_{\omega}(\mathbb{S}^{n-1})} P_{\alpha} [\tau_{x^{*}, \omega}](x) \\
	&  \leqslant C\|f\|_{\Lambda_{\omega}(\mathbb{S}^{n-1})} \omega\left(1-|x|\right)\\
&\leqslant C\|f\|_{\Lambda_{\omega}(\mathbb{S}^{n-1})} \omega(|x-y|),
\end{align*}
where $f_0(\zeta)=f(\zeta)-f(x^{*})$ for $\zeta \in \mathbb{S}^{n-1}$
and the doubling property of  $\omega$  is used in the last inequality. Similarly, we have
$$
\left|P_{\alpha} [f](y)-g(y)f\left(y^{*}\right)\right| \leqslant C\|f\|_{\Lambda_{\omega}(\mathbb{S}^{n-1})} \omega(|x-y|).
$$
Since  $\omega$ satisfies \eqref{eq-CCHL-02} by Lemma \ref{lem-omega}, it follows from
\cite[Theorem 2.3]{chen24} and the doubling property of $\omega$ that
\begin{align*}
	\left|g(x)f\left(x^{*}\right)-g(y)f\left(y^{*}\right)\right|
&\leq|g(x)|\cdot |f\left(x^{*}\right)-f\left(y^{*}\right)|+|f(y^{*})| \cdot |g(x)-g(y)|\\
	&\leq C\|f\|_{\Lambda_{\omega}(\mathbb{S}^{n-1})} \omega ( |x^{*}-y^{*} |)+C\|f\|_{\Lambda_{\omega}(\mathbb{S}^{n-1})}\omega(|x-y|) \\
	&\leq C\|f\|_{\Lambda_{\omega}(\mathbb{S}^{n-1})} \omega(|x-y|).
\end{align*}
Combining the above inequalities gives
\begin{align}\label{eq-case3.2}
\left|P_{\alpha} [f](x)-P_{\alpha} [f](y)\right| \leqslant C\|f\|_{\Lambda_{\omega}(\mathbb{S}^{n-1})}\omega(|x-y|).
\end{align}
Thus $(i)$ follows from \eqref{eq-case3.1} and \eqref{eq-case3.2}.
The proof of this theorem is complete.
\qed

\section{The Girela-Pel\'aez conjecture for the M\"obius invariant Laplacian operator}\label{Sec-3}

Before proving Theorem \ref{thm-1}, let us recall some well-known results.

\begin{Thm}{\rm (\cite[Theorem  4.1.1]{Sto-2016})}\label{Thm-J}
Suppose that $n\geq2$ and $\Psi$ is a twice continuously differentiable function of $\mathbb{B}^{n}$ into $\mathbb{R}$.
Then, for $r\in(0,1)$, \beqq
\int_{\mathbb{S}^{n-1}}\Psi(r\zeta)d\sigma(\zeta)=\Psi(0)+\int_{\mathbb{B}^{n}_{r}}g_{n}(|x|,r)\Delta_{h}\Psi(x)dV_{h}(x)
\eeqq and
\beqq
\frac{d}{dr}\int_{\mathbb{S}^{n-1}}\Psi(r\zeta)d\sigma(\zeta)=\frac{1}{n}r^{1-n}(1-r^{2})^{n-2}\int_{\mathbb{B}^{n}_{r}}\Delta_{h}\Psi(x)dV_{h}(x),
\eeqq
where $g_{n}$ is defined in Theorem \ref{thm-1} and $\mathbb{B}^{n}_{r}=\{x\in\mathbb{R}^{n}:~|x|<r\}$.
\end{Thm}

\begin{lem}\label{fp-inv-subh}
Let $p\in[2,\infty)$. Suppose that $\lambda_{1}$ and $\lambda_{2}$  are
 nonnegative continuous functions in $[0,1)$.
 If $f\in\mathbf{H}_{\lambda_{1},\lambda_{2}}(\mathbb{B}^{n})$,
 then $|f|^p$ is invariant subharmonic.
\end{lem}

\begin{proof}
Since
\beq\label{eq-fp-laph}
\Delta_{h}\left(|f(x)|^{p}\right)
=p(p-1)|f(x)|^{p-2}|\nabla^{h} f(x)|^{2}
+p|f(x)|^{p-2}f(x)\Delta_{h}f(x),
\eeq
we obtain $\Delta_{h}\left(|f(x)|^{p}\right)\geq 0$
from $p\geq 2$ and $f(x)\Delta_{h}f(x)\geq 0$.
This completes the proof.
\end{proof}

%

\subsection*{The proof of Theorem \ref{thm-1}}
Since $f\in\mathbf{H}_{\lambda_{1},\lambda_{2}}(\mathbb{B}^{n})$, from \eqref{eq-fp-laph},
we see that for $x\in\mathbb{B}^{n}$,
\beq\label{ef-01}
\Delta_{h}\left(|f(x)|^{p}\right)
&\leq&p(p-1)|f(x)|^{p-2}|\nabla^{h} f(x)|^{2}+p|f(x)|^{p-2}\left(\lambda_{1}(|x|)|\nabla^{h} f(x)|^{2}+\lambda_{2}(|x|)|f(x)|^{2}\right)\\ \nonumber
&=&p\big(p-1+\lambda_{1}(|x|)\big)|f(x)|^{p-2}|\nabla^{h} f(x)|^{2}+p\lambda_{2}(|x|)|f(x)|^{p}.
\eeq
Then by Theorem \ref{Thm-J} and (\ref{ef-01}), we have
\beq\label{ef-03}
M_{p}^{p}(r,f)&=&|f(0)|^{p}+\int_{\mathbb{B}^{n}_{r}}g_{n}(|x|,r)\Delta_{h}\big(|f(x)|^{p}\big)dV_{h}(x)\\ \nonumber
&\leq&|f(0)|^{p}+p\mathcal{J}_{1}(r)
+p\mathcal{J}_{2}(r),
\eeq
where
\beqq
\mathcal{J}_{1}(r)=\int_{\mathbb{B}^{n}_{r}}g_{n}(|x|,r)\left(p-1+\lambda_{1}(|x|)\right)|f(x)|^{p-2}|\nabla^{h} f(x)|^{2}dV_{h}(x)
\eeqq
and
\beqq
\mathcal{J}_{2}(r)=\int_{\mathbb{B}^{n}_{r}}\lambda_{2}(|x|)g_{n}(|x|,r)|f(x)|^{p}dV_{h}(x).
\eeqq

Next, we estimate $\mathcal{J}_{1}(r)$ and $\mathcal{J}_{2}(r)$.
For $\rho\in[0,r)$, the H\"older's inequality in the case $p>2$ yields
\be\label{ef-02}
\int_{\mathbb{S}^{n-1}}|f(\rho\zeta)|^{p-2}|\nabla f(\rho\zeta)|^{2}d\sigma(\zeta)\leq M_{p}^{2}(\rho,\nabla f)M_{p}^{p-2}(\rho,f).
\ee
Elementary calculations lead to
\beq\label{ef-04}
g_{n}(\rho,r)&=&\frac{1}{n}\int_{\rho}^{r}\frac{(1-s^{2})^{n-2}}{s^{n-1}}ds
\leq\frac{1}{n\rho^{n-1}}\int_{\rho}^{r}(1-s^{2})^{n-2}ds\\ \nonumber
&\leq&\frac{(1+r)^{n-2}}{n\rho^{n-1}}\int_{\rho}^{r}(1-s)^{n-2}ds\\ \nonumber
&\leq&\frac{2^{n-2}}{n(n-1)}\frac{(1-\rho)^{n-1}}{\rho^{n-1}}.
\eeq
By (\ref{eq-k-1}), (\ref{ef-02}), (\ref{ef-04}) and the inequality
\beqq
(1-\rho)\omega\left(\frac{1}{1-\rho}\right)\leq \omega(1),
\eeqq
we see that there is a positive constant $C$ such that
\beq\label{ef-07}
\mathcal{J}_{1}(r)&=&n\int_{0}^{r}\frac{\rho^{n-1}g_{n}(\rho,r)}{(1-\rho^{2})^{n-2}}\left(p-1+\lambda_{1}(\rho)\right)\left(\int_{\mathbb{S}^{n-1}}|f(\rho\zeta)|^{p-2}|\nabla f(\rho\zeta)|^{2}d\sigma(\zeta)\right)d\rho\\  \nonumber
&\leq&n\int_{0}^{r}\frac{\rho^{n-1}g_{n}(\rho,r)}{(1-\rho^{2})^{n-2}}\left(p-1+\lambda_{1}(\rho)\right)
M_{p}^{2}(\rho,\nabla f)M_{p}^{p-2}(\rho,f)d\rho\\  \nonumber
&\leq&\frac{2^{n-2}}{n-1}M_{p}^{p-2}(r,f)\int_{0}^{r}\left(p-1+\lambda_{1}(\rho)\right)(1-\rho)M_{p}^{2}(\rho,\nabla f)d\rho\\  \nonumber
&\leq&\frac{2^{n-2}C}{n-1}M_{p}^{p-2}(r,f)\int_{0}^{r}\left(p-1+\lambda_{1}(\rho)\right)(1-\rho)\left(\omega\left(\frac{1}{1-\rho}\right)\right)^{2}d\rho
\\  \nonumber
&\leq&\frac{2^{n-2}C\omega(1)}{n-1}M_{p}^{p-2}(r,f)\int_{0}^{r}\left(p-1+\lambda_{1}(\rho)\right)\omega\left(\frac{1}{1-\rho}\right)d\rho.
\eeq

Since $M_{p}^{p}(r,f)$ is non-decreasing with respect to $r\in[0,1)$ from Theorem \ref{Thm-J} and Lemma \ref{fp-inv-subh},   we see that
\beq\label{ef-08}
\mathcal{J}_{2}(r)&=&\int_{0}^{r}\frac{n\rho^{n-1}g_{n}(\rho,r)\lambda_{2}(\rho)}{(1-\rho^{2})^{n}}M_{p}^{p}(\rho,f)d\rho\\  \nonumber
&\leq&M_{p}^{p}(r,f)\int_{0}^{r}\frac{n\rho^{n-1}g_{n}(\rho,r)\lambda_{2}(\rho)}{(1-\rho^{2})^{n}}d\rho.
\eeq

Combining (\ref{ef-03}), (\ref{ef-07}) and (\ref{ef-08}) gives
\beqq
C(r)M_{p}^{p}(r,f)&\leq&|f(0)|^{p}+
\frac{2^{n-2}pC\omega(1)}{n-1}M_{p}^{p-2}(r,f)\int_{0}^{r}\left(p-1+\lambda_{1}(\rho)\right)\omega\left(\frac{1}{1-\rho}\right)d\rho,
\eeqq
which, together  with the non-decreasing property of $M_{p}^{p}(r,f)$, yields that
\beqq
C(r)M_{p}^{2}(r,f)&\leq&|f(0)|^{2}+
\frac{2^{n-2}pC\omega(1)}{n-1}\int_{0}^{r}\left(p-1+\lambda_{1}(\rho)\right)\omega\left(\frac{1}{1-\rho}\right)d\rho.
\eeqq
The proof of this theorem is complete.
\qed

\bigskip

{\bf Data Availability} Our manuscript has no associated data.\\

{\bf Conflict of interest} The authors declare that they have no conflict of interest.

\bigskip

\section*{Acknowledgments}
The research of The frist author is supported by Changsha Municipal Natural Science Foundation (grant no. kq2502155),
the Scientific Research Fund of Hunan Provincial Education Department (grant no. 25A0086),
the exchange program for the 4th meeting of the China-Montenegro Science and Technology Cooperation Committee (grant no. 4-4)
 and the construct program of the key discipline in Hunan Province.
The second author was partly supported by the
National Science Foundation of China (grant no. 12571080) and Guangxi Natural Science Foundation $\#$2026GXNSFFA00640002.
The third author is partially supported by JSPS KAKENHI (grant no. JP22K03363).
The  fourth author is supported by National Science Foundation of China
	(grant no. 12371071 and  12571081).


\begin{thebibliography}{99}
	\bibitem{Ai2002}
	\textsc{H. Aikawa:}
	\emph{H\"older continuity of the {D}irichlet solution for a general domain},
	{Bull. Lond. Math. Soc.},
	\textbf{34} (2002), 691--702.


	
		\bibitem{Ai2010}
		\textsc{H. Aikawa:}
		\emph{Modulus of continuity of the Dirichlet solutions},
		{Bull. Lond. Math. Soc.},
		\textbf{42} (2010), 857--867.

\bibitem{Ar} \textsc{G. Arfken:} \emph{Mathematical Methods for Physicists}, 3rd edn. Academic Press, Orlando, FL (1985).

\bibitem{AG} \textsc{D. H. Armitage  and S. J. Gardiner:}
 \textit{Classical potential theory}, Springer,    2000.

\bibitem{A-P}
\textsc{K. Astala and L. P\"aiv\"arinta:}
\emph{Calder\'on's inverse conductivity problem in the plane},
{Ann. Math.}, \textbf{163} (2006), 265--299.

\bibitem{ABR-2001} \textsc{S. Axler, P. Bourdon and W. Ramey:}
 \textit{Harmonic function theory}, Springer-Verlag, New York,   2001.
		




\bibitem{Bea} \textsc{ A. F. Beardon:}
\textit{The geometry of discrete groups}, Springer-Verlag, New York
Inc. 1983.


\bibitem{BS} \textsc{S. Bergman and M. Schiffer:} \emph{Kernel Functions and Elliptic Differential Equation in Mathematical Physics
Pure and Applied Mathematics.} Academic Press, New York (1953).

		\bibitem{AH2014}
		\textsc{A. Borichev and H. Hedenmalm:}
		\emph{Weighted integrability of polyharmonic functions},
		{Adv. Math.},
		\textbf{264} (2014), 464--505.





		
	\bibitem{chen23}
	\textsc{J. L. Chen, S. L. Chen, M. Z. Huang and H. Q. Zheng:}
	\emph{Isoperimetric type inequalities for mappings induced by
		weighted {L}aplace differential operators},
	{J. Geom. Anal.},
	\textbf{33} (2023), 45 pp.
	
	\bibitem{chen2021}
	\textsc{J. L. Chen, M. Z. Huang, S. Lee  and X. T. Wang:}
	\emph{Equivalent norms of solutions to hyperbolic Poisson's equations},
	{J. Geom. Anal.},
	\textbf{31} (2021), 8173--8201.
	
	\bibitem{chen2018}
	\textsc{J. L. Chen, M. Z. Huang, A. Rasila and X. T. Wang:}
	{\it On Lipschitz continuity of solutions of hyperbolic Poisson's equation},
	Calc. Var. Partial Differential Equations, \textbf{57} (2018), 32 pp.


\bibitem{C2026}
	\textsc{S. L.  Chen:}
	\emph{\it The Littlewood-Paley type $g$-function, the Lusin type $s$-function, and their applications},
	{Adv. Math.},  \textbf{496} (2026),  110993.
	
		\bibitem{chen24B}
	\textsc{S. L.  Chen and H. Hamada:}
	\emph{\it Equivalent norms, Hardy-Littlewood-type theorems,
		and their applications},
	{Sci. China Math.},  \textbf{68} (2025), 533--558.
	
	\bibitem{CH2025MathZ}
     \textsc{S. L.  Chen and H. Hamada:}
     \emph{\it Characterizations of  Hardy  spaces and composition operators in bounded symmetric domains},
     Math. Z.,
     \textbf{311} (2025), Paper No. 16, 24 pp.

	\bibitem{chen24}
	\textsc{S. L.  Chen, H. Hamada and D. Xie:}
	\emph{\it Hardy-Littlewood type theorems and a Hopf type
		lemma},
	{J. Geom. Anal.}, \textbf{34} (2024), 23 pp.


	
	
\bibitem{CPR} \textsc{S. L. Chen,  S. Ponnusamy, and A. Rasila:}
\emph{On characterizations of Bloch-type, Hardy-type, and Lipschitz-type spaces},
{Math. Z.}, {\bf 279} (2015), 163--183.


\bibitem{CRW} \textsc{S. L. Chen, A. Rasila, and X. Wang:}
\emph{Radial growth, Lipschitz and Dirichlet spaces on solutions to the non-homogenous Yukawa equation},
Israel J. Math., {\bf 204} (2014), 261--282.


\bibitem{CS-2015}
	\textsc{S. L.   Chen and Z. H. Su:}
Radial growth and Hardy-Littlewood-type theorems on hyperbolic harmonic functions,
{Filomat}, {\bf 29} (2015), 361--370.

	
	\bibitem{chen15}
	\textsc{S. L.   Chen and M. Vuorinen:}
	\emph{Some properties of a class of elliptic partial differential
		operators},
	{J. Math. Anal. Appl.},
	\textbf{431} (2015), 1124--1137.
	
	
	\bibitem{DF16}
	\textsc{P. Diaconis and P. Forrester:}
	\emph{Hurwitz and the origins of random matrix theory in mathematics},
	{Random Matrices Theory Appl.}, 	\textbf{6} (2017),
	26 pp.
	

	
	\bibitem{DK1997}
	\textsc{K. Dyakonov:}
	\emph{Equivalent norms on Lipschitz-type spaces of holomorphic functions},
	{Acta Math.}, 	\textbf{178} (1997),
	143--167.
	
    \bibitem{DK04}
    \textsc{K. Dyakonov:}
    \emph{Holomorphic functions and quasiconformal mappings with smooth moduli},
   {Adv. Math.}, 	\textbf{187} (2004),
   146--172.
	
	


 \bibitem{DK06}
    \textsc{K. Dyakonov:} \emph{Addendum to ``Strong Hardy-Littlewood theorems for
analytic functions and mappings of finite distortion''},  {Math. Z.}, 	\textbf{254} (2006),
   433--437.



\bibitem{GPP} \textsc{D. Girela, M. Pavlovic and J. A. Pel\'{a}ez:}
\emph{Spaces of analytic functions of Hardy-Bloch type}, {J. Anal.
Math.} {\bf 100}(2006), 53--81.

\bibitem{GP} \textsc{D. Girela and J. A. Pel\'{a}ez:}
\emph{Integral means of analytic functions}, Ann. Acad. Sci. Fenn.  Math. {\bf 29}(2004), 459--469.





	
	\bibitem{HL31}
	\textsc{G. Hardy and J. Littlewood:}
	\emph{Some properties of conjugate functions},
	{J. Reine Angew. Math.},
	\textbf{167} (1932),  405--423.
	
	
	\bibitem{HL32}
	\textsc{G. Hardy and J. Littlewood:}
	\emph{Some properties of fractional integrals. {II}},
	{Math. Z.},
	\textbf{34} (1932),  403--439.
	
\bibitem{HZ}  \textsc{E. Heinz:}
\emph{On certain nonlinear elliptic differential equations and univalent mappings},
J. Anal. Math., {\bf 5}(1956/57), 197--272.



	\bibitem{Hur1897}
	\textsc{A. Hurwitz:}
	\emph{\it  \"{U}ber die Erzeugung der invarianten durch integration},
	{ Nachr. Ges. Wiss.}, G\"{o}ttingen, \textbf{1897} (1897), 71--90.
	
	\bibitem{It2012}
	\textsc{T. Itoh:}
	\emph{\it  Modulus of continuity of {$p$}-{D}irichlet solutions in a metric measure space},
	{Ann. Acad. Sci. Fenn. Math.}, \textbf{37} (2012), 339--355.
	
	\bibitem{KMM2021}
	\textsc{A. Khalfallah, M.  Mateljevi\'c and M. Mhamdi:}
	\emph{Some properties of mappings admitting general {P}oisson
		representations},
	{Mediterr. J. Math.},
	\textbf{18} (2021),  19 pp.
	
	
	
	
\bibitem{Kur} \textsc{\"U. Kuran:} \emph{Subharmonic behavior of $|h|^{p}~(p>0, h~\mbox{harmonic})$},
{J. London Math. Soc.}, {\bf 8} (1974), 529--538.



	

	
	\bibitem{leu} 	\textsc{H. Leutwiler:}
	\emph{Best constants in the Harnack inequality for the Weinstein equation},
	{Aequationes Math.},
	\textbf{34} (1987), 304--315.
	
	\bibitem{li24}
	\textsc{Q. Li and J. Chen:}
	\emph{Schwarz's Lemma for the solutions to the Dirichlet
		problems for the invariant Laplacians},
	{Bull. Malays. Math. Sci. Soc.}, 	\textbf{47} (2024),
	21 pp.
	
	\bibitem{Liu04}
	\textsc{C. Liu and L. Peng:}
	\emph{Boundary regularity in the Dirichlet problem for the invariant {L}aplacians {$\Delta_\gamma$} on the unit real ball},
	{Proc. Amer. Math. Soc.},
	\textbf{132} (2004), 3259--3268.
	
	\bibitem{Liu09}
	\textsc{C. Liu and L. Peng:}
	\emph{Generalized {H}elgason-{F}ourier transforms associated to variants of the {L}aplace-{B}eltrami operators on the unit ball in {$\Bbb R^n$}},
	{Indiana Univ. Math. J.},
	\textbf{58} (2009), 1457--1491.
	
	\bibitem{Liu21}
	\textsc{C. Liu, A. Per\"{a}l\"{a} and J. Si:}
	\emph{Weighted integrability of polyharmonic functions in the higher-dimensional case},
	{Anal. PDE},
	\textbf{14} (2021), 2047--2068.
	
	\bibitem{Liu24}
	\textsc{C. Liu and H. Xu:}
	\emph{Lipschitz continuity of the solutions to the Dirichlet problems for the invariant Laplacians},
	{J. Math. Anal. Appl.},
	\textbf{538} (2024), 11 pp.
	

\bibitem{Mar}  \textsc{G. J. Martin:}
\emph{Harmonic degreen 1 maps are diffeomorphisms: Lewy's theorem for curved metrics},
{Trans. Amer. Math. Soc.}, {\bf 368} (2016), 647--658.


	\bibitem{Ni}   \textsc{J. C. C. Nitsche:} \emph{The boundary behavior of minimal surfaces.
Kellogg's theorem and branch points on the boundary}, {Inven. Math.},
 {\bf 8} (1969), 313--333.


	
	\bibitem{NO88}
	\textsc{C. Nolder and D. Oberlin:}
	\emph{Moduli of continuity and a {H}ardy-{L}ittlewood theorem},
	Complex analysis, Joensuu 1987, 265--272.
Lecture Notes in Math., 1351.
Springer-Verlag, Berlin, 1988.
	



	\bibitem{Ol14}
	\textsc{A. Olofsson:}
	\emph{Differential operators for a scale of {P}oisson type kernels
		in the unit disc},
	{J. Anal. Math.},
	\textbf{123} (2014), 227--249.
	
	\bibitem{Ol20}
	\textsc{A. Olofsson:}
	\emph{On a weighted harmonic Green function and a theorem of
		{L}ittlewood},
	{Bull. Sci. Math.},
	\textbf{158} (2020), 63 pp.


\bibitem{Pav}
	\textsc{M. Pavlovi\'c:}
	\emph{Lipschitz conditions on the modulus of a harmonic function},
	{Rev. Mat. Iberoam.},
	\textbf{23} (2007), 831--845.
	

\bibitem{Qiu} \textsc{G. H. Qiu:} \emph{Interior curvature estimates for hypersurfaces of prescribing scalar curvature in dimension three},
{Amer. J. Math.}, {\bf 146} (2024), 579--605.

\bibitem{SS-05} \textsc{S. Stevi\'c:} \emph{Area  type inequalities and integral means of harmonic functions on the unit ball}, {J. Math. Soc. Japan},
{\bf 59} (2007), 583--601.

\bibitem{Sto-2016} \textsc{M. Stoll:}
\emph{Harmonic and subharmonic function theory on the hyperbolic ball}, Cambridge University Press, Cambridge, 2016.

	
	\bibitem{U} \textsc{D. Ullrich:} \emph{Radial limits of $M$-subharmonic functions},
{Trans. Am.  Math. Soc.},
{\bf 292} (1985), 501--518.



\bibitem{Wan} \textsc{M. T. Wang:} \emph{Interior gradient bounds for solutions to the minimal surfacesystem},
{Amer. J. Math.},
{\bf 126} (2004), 921--934.
	
	\bibitem{ZHD24}
	\textsc{D. Zhong, M. Huang and D. Wei:}
	\emph{Some inequalities for self-mappings of unit ball satisfying
		the invariant {L}aplacians},
	{Monatsh. Math.}, \textbf{203} (2024), 911--925.
	
	\bibitem{zhou22}
	\textsc{L. Zhou:}
	\emph{A {B}ohr phenomenon for {$\alpha$}-harmonic functions},
	{J. Math. Anal. Appl.}, \textbf{505} (2022), 21 pp.
	
	
	
	\bibitem{Z59}
	\textsc{A. Zygmund:}
	\emph{Trigonometric series},
Cambridge University Press, New York, 1959.
	
	\end{thebibliography}
\end{document}